\documentclass{daj}

\dajAUTHORdetails{%
  title = {Uniform Estimates for Smooth Polynomials over Finite Fields}, 
  author = {Ofir Gorodetsky},
  plaintextauthor = {Ofir Gorodetsky},
  keywords = {smooth polynomials, smooth permutations, Dickman function, largest prime factor, saddle point analysis},
}   

\dajEDITORdetails{%
   year={2023},
   number={16},
   received={16 March 2022},   
   revised={18 July 2023},    
   published={4 October 2023},  
   doi={10.19086/da.87773},       
}   

\usepackage{amssymb}
\usepackage{amsfonts}
\usepackage{amsthm}
\usepackage{amsmath}
\usepackage{mathtools}
\usepackage{url}

\newtheorem{thm}{Theorem}[section]

\newtheorem{lem}[thm]{Lemma}  
\newtheorem{prop}[thm]{Proposition}
\newtheorem{cor}[thm]{Corollary}

\newcommand{\RR}{\mathbb{R}}
\newcommand{\PP}{\mathbb{P}}
\newcommand{\CC}{\mathbb{C}}
\newcommand{\NN}{\mathbb{N}}

\newcommand{\FF}{\mathbb{F}}
\newcommand{\EE}{\mathbb{E}}

\newcommand{\parity}{a}

\newcommand{\Mnq}{\mathcal{M}_{n,q}}

\mathtoolsset{showonlyrefs}

\numberwithin{equation}{section}

\begin{document}

\begin{frontmatter}[classification=text]

\title{Uniform Estimates for Smooth Polynomials over Finite Fields} 

\author[ofir]{Ofir Gorodetsky\thanks{Supported by the European Research Council (ERC) under the European Union's Horizon 2020 research and innovation programme (grant agreements Nos. 786758 and 851318)}.}

\begin{abstract}
We establish new estimates for the number of $m$-smooth polynomials of degree $n$ over a finite field $\FF_q$, where the main term involves the number of $m$-smooth permutations on $n$ elements.

Our estimates imply that the probability that a random polynomial of degree $n$ is $m$-smooth is asymptotic to the probability that a random permutation on $n$ elements is $m$-smooth, uniformly for $m\ge (2+\varepsilon)\log_q n$ as $q^n \to \infty$. This should be viewed as an unconditional analogue of works of Hildebrand and of Saias in the integer setting, which assume the Riemann Hypothesis. Moreover, we show that the range $m \ge (2+\varepsilon)\log_q n$ is sharp; this should be viewed as a resolution of a (polynomial analogue of a) conjecture of Hildebrand.

As an application of our estimates, we determine the rate of decay in the asymptotic formula for the expected degree of the largest prime factor of a random polynomial.
\end{abstract}
\end{frontmatter}

\section{Introduction}
Given a positive integer $n$, we let $\pi_n$ be a permutation chosen uniformly at random from $S_n$. Given a prime power $q$, we let $f_n=f_{n,q} \in \FF_q[T]$ be a polynomial chosen uniformly at random from  $\Mnq\subseteq \FF_q[T]$, the set of monic polynomials of degree $n$ over the finite field $\FF_q$.

We say that a permutation is $m$-smooth if all its cycles are of length at most $m$. Similarly, we say that a polynomial is $m$-smooth if all its prime (i.e.~irreducible) factors are of degree at most $m$. We define
\begin{equation}
	\psi_{\pi}(n,m) := \#\{ \pi \in S_n: \pi \text{ is }m\text{-smooth}\}, \qquad \psi_q(n,m) := \#\{ f \in \Mnq: f\text{ is }m\text{-smooth}\},
\end{equation}
so that $\PP(\pi_n \text{ is }m\text{-smooth})=\psi_{\pi}(n,m)/n!$ and $\PP(f_n \text{ is }m\text{-smooth}) = \psi_q(n,m)/q^n$.

Throughout the paper, $n \ge 2$, $1 \le m \le n$ and
\begin{equation}
	u := \frac{n}{m}.
\end{equation}
Smoothness probabilities in $S_n$ were studied extensively in the literature, dating back to Goncharov \cite{goncharov1944}, who studied $\PP(\pi_n \text{ is }m\text{-smooth})$ in the bounded $u$ regime; see \cite{shepp1966, arratia2003} for generalizations. The bounded $m$ regime received attention as well \cite{chowla1951, moser1955, wimp1985}. Asymptotic results covering the entire range of $m$ were established by Manstavi\v{c}ius and Petuchovas \cite{manstavicius2016}. See also the recent works of Ford \cite{ford2021cycle} and the author \cite{revisited}.

Smoothness probabilities in $\Mnq$ were studied extensively as well, starting with the work of Odlyzko \cite[App.~1]{odlyzko1985}, which was motivated by cryptography (see Pomerance's survey for the role of smooth numbers in cryptography \cite{pom}). Subsequent works on $\psi_q$ include \cite{warlimont1991, lovorn1992, manstavicius1992,manstavicius19922, panario1998, lovorn1998, soundararajan, arratia2003,joux2006} and are expanded upon later.

Here we estimate smoothness probabilities in $\mathcal{M}_{n,q}$ by comparing them to the corresponding probabilities in $S_n$, uniformly in the parameters $n$ and $q$. Unless stated otherwise, constants, both implicit and explicit, are absolute. For recent results where new function field estimates are obtained by comparing to a permutation quantity, see \cite{gorodetsky2017, elboim2020uniform}.
\begin{thm}\label{thm:laplace}
	If $m \ge 6\log n$, then
	\begin{equation}\label{eq:main largen}
		\frac{\PP(f_n \text{ is }m\text{-smooth})}{\PP(\pi_n \text{ is }m\text{-smooth})} =1+ O\left(\frac{u \log (u+1)}{ mq^{\lceil \frac{m+1}{2}\rceil}}\right).
	\end{equation}
\end{thm}
The previous asymptotic results for $\PP(f_n \text{ is }m\text{-smooth})$ in this range of $m$ achieved only the weaker error terms $u \log (u+1)/m$ (see \eqref{eq:manstaerror} and \eqref{eq:sound1}) or $1/u$ (see \S\ref{sec:saddle}), but with perhaps simpler main terms. For smaller $m$, we prove
\begin{thm}\label{thm:saddle}
	Fix $\varepsilon>0$. If $8\log n\ge m \ge (2+\varepsilon)\log_q n$, then
	\begin{equation}\label{eq:main mediumn2}
		\frac{\PP(f_n \text{ is }m\text{-smooth})}{\PP(\pi_n \text{ is }m\text{-smooth})} = 1 + O_\varepsilon\left(  \frac{u n^{\frac{1+\parity}{m}}}{q^{\lceil \frac{m+1}{2} \rceil}}\right)
	\end{equation}
	where $\parity = \mathbf{1}_{2 \mid m}$. 
\end{thm}
Here and throughout, $\log_q$ is the base-$q$ logarithm.
Once $m \ge 4\log_q n$, the error term in \eqref{eq:main mediumn2} decays faster than $O(1/u)$, the existing error term in the asymptotic result for $\PP(f_n \text{ is }m\text{-smooth})$ in this range of $m$ (see \S\ref{sec:saddle}). In Theorem~\ref{thm:saddle2} below, the error term in both theorems is improved, at least if $m$ is not too close to $n$, by modifying the main term $1$ in the right-hand side of \eqref{eq:main mediumn2}.

Theorems~\ref{thm:laplace} and \ref{thm:saddle} cover the entire range $n \ge m \ge  (2+\varepsilon)\log_q n$, and in particular show that
\begin{equation}\label{eq:akin}
	\PP(f_n \text{ is }m\text{-smooth}) \sim \PP(\pi_n \text{ is }m\text{-smooth})
\end{equation}
holds as $q^n \to \infty$, uniformly in that range (a short computation shows that the relative error term in \eqref{eq:akin} is $\ll_{\varepsilon} 1/(nq)^{c_{\varepsilon}}$).  
When $n \to \infty$ and $m \le (1-\varepsilon)\log_q n$, the probabilities are no longer of the same order of magnitude. Concretely, for $n$ tending to $\infty$ and fixed $q$, we have
\begin{equation}
	\PP(f_n \text{ is }m\text{-smooth}) = q^{-n+o_{q,\varepsilon}(n)}
\end{equation}
for $m \sim (1-\varepsilon) \log_q n$ by \cite[Thm.~1.4]{soundararajan}, while, in the same limit,
\begin{equation}
	\PP(\pi_n \text{ is }m\text{-smooth}) = q^{-\frac{n}{1-\varepsilon} + o_{q,\varepsilon}(n)}
\end{equation}
by Proposition~\ref{prop:perm}. For $m=o(\log n)$ the situation is even worse: $\PP(\pi_n \text{ is }m\text{-smooth})$ decays superexponentially in $n$ by Proposition~\ref{prop:perm}, while $\PP(f_n \text{ is }m\text{-smooth}) \ge q^{-n}$ can never decay superexponentially.

We prove a comparative result where $m$ can be as small as $(1+\varepsilon) \log_q n$. It requires the introduction of some notation. Define $x=x_{n,m}>0$ by 
\begin{equation}\label{eq:x def}
	\sum_{j=1}^{m} x^j = n.
\end{equation}
Observe that $1 \le x \le n^{1/m}$. In fact, from Lemma~\ref{lem:lambda} it follows that $x=\Theta(n^{1/m})$. Let
\begin{equation}
	G_q(z) := \prod_{\substack{P \in \mathcal{P} \\ \deg (P) \le m}}\left(1-\left(\frac{z}{q}\right)^{\deg (P)}\right)^{-1} \exp\left(-\sum_{i=1}^{m} \frac{z^i}{i}\right),
\end{equation}
where $\mathcal{P}=\mathcal{P}_{q}$ is the set of monic irreducible polynomials over $\FF_q$. By \eqref{eq:gqx}, $G_q(x) \ge 1$.
\begin{thm}\label{thm:saddle2}
	Fix $\varepsilon > 0$.	If $(2+\varepsilon)\log_q n \ge m \ge (1+\varepsilon)\log_q n$, then
	\begin{equation}\label{eq:main mediumn3}
		\frac{\PP(f_n \text{ is }m\text{-smooth})}{\PP(\pi_n \text{ is }m\text{-smooth})} =  G_q(x) \left( 1+ O_{\varepsilon}\left( m\left( \frac{1}{q^{m/2}}+\frac{n^3}{q^{2m}}\right)\right)\right)
	\end{equation}
	where $x$ is defined in \eqref{eq:x def}.
	If $n/(\log n \log^3 \log(n+1)) \ge m \ge (2+\varepsilon)\log_q n$, then $1\le G_q(x) \le C_{\varepsilon}$ and
	\begin{equation}\label{eq:main mediumn32}
		\frac{\PP(f_n \text{ is }m\text{-smooth})}{\PP(\pi_n \text{ is }m\text{-smooth})} =  G_q(x)  + O_{\varepsilon}\left(\frac{n^{\frac{1+\parity}{m}} \min\{m, \log u\}}{mq^{\lceil \frac{m+1}{2}\rceil}}  \right)
	\end{equation}
	where $\parity = \mathbf{1}_{2 \mid m}$.
\end{thm}
Estimate \eqref{eq:main mediumn3} gives an asymptotic result as $n \to \infty$ if $m \ge (3/2+\varepsilon)\log_q n \to \infty$, and in any case gives a non-trivial upper bound. Estimate \eqref{eq:main mediumn32} gives an asymptotic result as $q^n \to \infty$. Previous asymptotic results for the range of \eqref{eq:main mediumn32} achieved only the weaker error terms $u \log (u+1)/m$ (see \eqref{eq:manstaerror} and \eqref{eq:sound1}) or $1/u$ (see \S\ref{sec:saddle}). 

In \cite{revisited} we show the condition $n/(\log n \log ^3 \log (n+1)) \ge m$ may be omitted from \eqref{eq:main mediumn32}, and that a different asymptotic result holds for $m \le (3/2-\varepsilon)\log_q n$.

It is natural to ask when is the function $G_q(x)$ asymptotic to $1$. This is answered in the following theorem.
\begin{thm}\label{thm:Gqsize}
	If $(m - 2\log_q n)\log q \to \infty$ then $G_q(x)$ tends to $1$, and so in particular \eqref{eq:akin} holds in this limit. If  $m=2\log_q n +O(1)$ then $G_q(x)-1 \gg_q 1$. If we let $m/\log_ q n$ tend to $2-\varepsilon$ ($\varepsilon \in (0,1)$) then $\log G_q(x) = \Theta_q(n^{\varepsilon+o(1)})$.
\end{thm}
We conclude with  the following uniform result, holding in the entire range, but most useful when $u$ is bounded.
\begin{prop}\label{prop:bndd unif}
	For $n \ge m \ge 1$,
	\begin{equation}
		0 \le \PP(f_n \text{ is }m\text{-smooth}) - \PP(\pi_n \text{ is }m\text{-smooth}) \le \frac{C}{m q^{\lceil \frac{m+1}{2} \rceil}}.
	\end{equation}
\end{prop}
We stress that our proofs of all of our estimates are direct, in the sense that we do not make use of existing asymptotics of $\PP(f_n \text{ is }m\text{-smooth})$. Rather, we bound the differences $|\PP(f_n \text{ is }m\text{-smooth})-\PP(\pi_n \text{ is }m\text{-smooth})|$ and $|\PP(f_n \text{ is }m\text{-smooth})-G_q(x)\PP(\pi_n \text{ is }m\text{-smooth})|$ themselves. In the very last step of each proof we plug existing lower bounds on $\PP(\pi_n \text{ is }m\text{-smooth})$ to convert our bounds to relative bounds. We explain the proof strategy in more detail in \S\ref{sec:strat}.
\subsection{Optimality}
Theorems~\ref{thm:laplace} and \ref{thm:saddle} combined state
\begin{equation}\label{eq:combined}
	\frac{\PP(f_n \text{ is }m\text{-smooth})}{\PP(\pi_n \text{ is }m\text{-smooth})} =1+ O_{\varepsilon}\left(\frac{u n^{\frac{1+\parity}{m}} \min\{ m ,\log(u+1)\}}{m q^{\lceil \frac{m+1}{2}\rceil}}\right)
\end{equation}
for $m \ge (2+\varepsilon)\log_q n$.  This is essentially optimal in both the error term and the range. Indeed, from \eqref{eq:main mediumn32} and the lower bound for $G_q(x)$ given in \eqref{eq:gqx}, we have a matching lower bound as long as $(2+\varepsilon)\log_q n \le m \le n/(\log n\log^3 \log(n+1))$ and $n \gg 1$.  

We observe that for $m=n-1$, \eqref{eq:combined} recovers the Prime Polynomial Theorem with squareroot error term. Indeed,
\begin{equation}
	\begin{split}
		\PP(f_n \text{ is }(n-1)\text{-smooth}) &= 1-  \PP(f_n \text{ is irreducible}), \\\PP(\pi_n \text{ is }(n-1)\text{-smooth}) &= 1-\PP(\pi_n \text{ is an }n\text{-cycle}) = 1-\frac{1}{n},
	\end{split}
\end{equation}
and so from \eqref{eq:combined}
\begin{equation}\label{eq:ppt}
	\#\{ f \in \mathcal{M}_{n,q} : f \text{ is irreducible}\} = \frac{q^n}{n} + O\left(\frac{q^{\lfloor n/2\rfloor}}{n}\right).
\end{equation}
In view of Theorem~\ref{thm:Gqsize}, the range of Theorem~\ref{thm:saddle} cannot be extended to $m \ge 2\log_q n$. (We can extend it to $m-2\log_qn \to \infty$ but with a worse error term.)
\subsection{Expected largest prime factor}
Let $L(\pi)$ be the size of the longest cycle in the cycle decomposition of $\pi$. Similarly, let $L_{q}(f)$ be the largest degree of a prime polynomial dividing $f \in \FF_q[T]$.

Golomb \cite{golomb1997} proved that $\EE L(\pi_n)/n$ tends to a limit, known as the Golomb--Dickman constant, and approximated it as $0.624329\ldots$. Knopfmacher and Manstavi\v{c}ius \cite{knopfmacher1997} proved that 
\begin{equation}\label{eq:knopf}
\EE L_q(f_n) - \EE L(\pi_n) = O\left( \sqrt{\frac{n}{q\log n}}\right)
\end{equation}
holds uniformly in $n$ and $q$. We prove the following estimate, which uncovers a transition around $\log q\asymp n\log n$.
\begin{thm}\label{thm:golomb}
	We have
	\begin{equation}
		c\exp(-C \max\{\sqrt{n\log n \log q}, \log q\}) \le  \EE L(\pi_n)-\EE L_q(f_n)  \le C \exp(-c \max\{ \sqrt{n\log n \log q}, \log q\}).
	\end{equation}
\end{thm}
To relate $L$ and $L_q$ to smooth permutations and polynomials, observe that $\PP(L(\pi_n) \le m)$ is equal to $\PP(\pi_n \text{ is }m\text{-smooth})$ and $\PP(L_q(f_n) \le m)$ is equal to $\PP(f_n \text{ is }m\text{-smooth})$. The identity $\EE X = 1+\sum_{i \ge 1}  (1-\PP(X \le i))$ for $\NN$-valued random variables shows that
\begin{equation}\label{eq:iden expected}
\EE L(\pi_n)-	\EE L_q(f_n)  = \sum_{m=1}^{n}  \left(  \PP(f_n \text{ is }m\text{-smooth})-\PP(\pi_n \text{ is }m\text{-smooth}) \right).
\end{equation}
By definition, the total variation distance between $L_q(f_n)$ and $L(\pi_n)$ is
\begin{multline}
	d_{TV}(L_q(f_ n), L(\pi_n)) = \sum_{m=1}^{n} | \PP(f_n \text{ is }m\text{-smooth}) - \PP(\pi_n \text{ is }m\text{-smooth}) \\- (\PP(f_n \text{ is }(m-1)\text{-smooth}) - \PP(\pi_n \text{ is }(m-1)\text{-smooth}))|.
\end{multline}
The fact that $\PP(f_n \text{ is }m\text{-smooth}) \ge \PP(\pi_n \text{ is }m\text{-smooth})$ (see Proposition~\ref{prop:bndd unif}) implies that
\begin{equation}
	\frac{1}{2}d_{TV}(L_q(f_n), L(\pi_n)) \le \EE L(\pi_n)-\EE L_q(f_n) ,
\end{equation}
so Theorem~\ref{thm:golomb} yields an upper bound for $d_{TV}(L_q(f_n), L(\pi_n))$:
\begin{cor}
	We have
	\[d_{TV}(L_q(f_n), L(\pi_n)) \le C \exp(-c \max\{ \sqrt{n\log n \log q}, \log q\}).\]
\end{cor} 
\subsection{Previous works}
\subsubsection{Dickman function and smoothness}
To discuss previous results, we introduce the Dickman function $\rho \colon [0,\infty) \to (0,1]$. It is defined as $\rho(t)=1$ for $t \in [0,1]$, and for the rest of its range it is defined through the delay differential equation
\begin{equation}
	t\rho'(t)+\rho(t-1)=0.
\end{equation}
It is a weakly-decreasing function, that decays superexponentially: $\rho(t)=t^{-t+o(t)}$ \cite[Eq.~(1.7)]{hildebrand1993}. It was introduced by Dickman \cite{dickman1930} in his study of smooth integers. We say that a positive integer is $y$-smooth if its prime factors are no larger than $y$. Dickman proved that, for any fixed $a>0$,
\begin{equation}\label{eq:dickman}
	\#\{ 1 \le n \le x: n\text{ is }x^{1/a}\text{-smooth}\} \sim x\rho(a)
\end{equation}
as $x \to \infty$. Goncharov \cite{goncharov1944} proved that
\begin{equation}\label{eq:gonch}
	\PP(\pi_n \text{ is }m\text{-smooth}) =\rho(u)+o(1)
\end{equation}
as $u=n/m$ tends to a positive constant; this is a permutation analogue of \eqref{eq:dickman}. From Proposition~\ref{prop:bndd unif} and Goncharov's result, we immediately obtain a polynomial analogue of his result:
\begin{equation}
	\PP(f_n \text{ is }m\text{-smooth}) =\rho(u)+o(1) + O\left(\frac{1}{mq^{\lceil \frac{m+1}{2}\rceil}}\right) = \rho(u)+o(1)
\end{equation}
as $n/m$ tends to a positive constant. This polynomial analogue was first established by Car \cite{car1987}, without making use of Goncharov's work. Warlimont later proved the quantitative estimate \cite{warlimont1991}
\begin{equation}\label{eq:warlimont}
	\PP\left( f_n \text{ is }m\text{-smooth}\right) = \rho(u) + O\left( \frac{1}{m} \right).
\end{equation}
A slightly weaker version of \eqref{eq:warlimont} was proved independently by Panario, Gourdon and Flajolet \cite{panario1998}.

We explain how \eqref{eq:warlimont} easily follows from our work. We first state a quantitative version of \eqref{eq:gonch}, proved in \cite{manstavicius2016}: 
\begin{equation}\label{eq:pn transform}
	\PP(\pi_n \text{ is }m\text{-smooth}) = \rho(u) \left(1+O\left( \frac{u \log (u+1)}{m}\right) \right)
\end{equation}
for $m \ge \sqrt{n \log n}$.	For $m \ge \sqrt{n \log n}$, \eqref{eq:warlimont} follows at once from Proposition~\ref{prop:bndd unif} and \eqref{eq:pn transform}, with an improved error term.
For $m < \sqrt{n\log n}$, $1/m$ is greater than $c\rho(u)$, so \eqref{eq:warlimont} amounts to $\PP(f_n \text{ is }m\text{-smooth})=O(1/m)$, which follows from Proposition~\ref{prop:bndd unif} and \eqref{eq:pn transform}.

For $m \gg \sqrt{n \log n}$, Manstavi\v{c}ius proved \cite[Thm.~2]{manstavicius1992}
\begin{equation}\label{eq:manstaerror}
	\PP\left( f_n \text{ is }m\text{-smooth}\right) = \rho(u) \left( 1 +O\left(\frac{u\log(u+1)}{m}\right)\right).
\end{equation}
This follows at once from \eqref{eq:combined} and \eqref{eq:pn transform}.

Under the Riemann Hypothesis, Hildebrand \cite{hildebrand1984} (cf.~\cite[Thm.]{granville2008}) proved that \begin{equation}\label{eq:hild}\#\{n \le x: n \text{ is }y\text{-smooth}\} = x\rho\left(\frac{\log x}{\log y}\right)\exp\left(O_{\varepsilon}\left(\frac{\log\left(\frac{\log x}{\log y}+1\right)}{\log y}\right)\right)
\end{equation}
uniformly for $\log y  \ge (2+\varepsilon)\log \log x$. In fact, this result also implies the Riemann Hypothesis. In the same range, Saias improved it to
\[\#\{n \le x: n \text{ is }y\text{-smooth}\} = \Lambda(x,y) \left( 1+ O_{\varepsilon}\left(\frac{\log x}{y^{1/2+\varepsilon}}\right)\right)\]
where $\Lambda(x,y)$ is a main term introduced by de Bruijn \cite{debruijn1951}. We find \eqref{eq:combined} akin to Saias's result, where in the polynomial setting, $\Lambda(x,y)/x$ is replaced with $\PP(\pi_n \text{ is }m\text{-smooth})$. This analogy led us in \cite{gorodetsky2022debruijn} to prove a version of Theorem \ref{thm:saddle2} in the integers, using a very similar proof method.

Hildebrand \cite{hildebrand1986} conjectured that the range of validity of \eqref{eq:hild} cannot be extended beyond $y \ge (\log x)^{2+\varepsilon}$. Theorems \ref{thm:saddle2} and \ref{thm:Gqsize} show \eqref{eq:akin} does not hold in the wider range $m \ge (2-\varepsilon)\log_q n$, which answers a version of Hildebrand's conjecture in the polynomial setting (as $q^n$ and $q^m$ correspond to $x$ and $y$ in integers). This observation inspired us in \cite{gorodetsky2022dickman} to resolve Hildebrand's conjecture in integers in an unconditional manner.

It is natural to ask whether the approximation $\PP(\pi_n \text{ is }m\text{-smooth}) \sim \rho(u)$ holds beyond the range $m/\sqrt{n\log n}\to \infty$ implied by \eqref{eq:pn transform}. In the appendix we show that $\rho$ is not an asymptotic approximation once $m/\sqrt{n \log n}$ is bounded from above. A related result was proved in \cite[Cor.~3]{manstavicius2016} by a different method.
\subsubsection{Saddle point analysis}\label{sec:saddle}
Odlyzko \cite{odlyzko1985} used saddle point analysis to estimate $\PP( f_n \text{ is }m\text{-smooth})$ for $q=2$ and $n^{1/100} \le m\le n^{99/100}$. Lovorn \cite{lovorn1992} extended this to general $q$ in her thesis. The range $n^{1/100} \le m\le n^{99/100}$ is captured in full by Theorem~\ref{thm:laplace}. 

Manstavi\v{c}ius \cite{manstavicius1992, manstavicius19922} extended the range of these results to $n \ge m\log m(\log \log m)^3$ with $m \to \infty$, and proved a relative error term of order $u^{-1} +mq^{-m}$. This result is an analogue of the work of Hildebrand and Tenenbaum \cite{hildebrand19862}, who estimated $\#\{1 \le n \le x : n \text{ is }y\text{-smooth}\}$  uniformly in the range $2 \le y \le x$, with a relative error term of $\log y/\log x + \log y/y$. We do not give the full statements of these asymptotics as they are somewhat complicated and are not needed here. For permutations, Manstavi\v{c}ius and Petuchovas obtained \cite[Thm.~2,~Cor.~5]{manstavicius2016}
\begin{equation}\label{eq:pn saddle}
	\PP(\pi_n \text{ is }m\text{-smooth}) = \frac{D(x)}{\sqrt{2\pi \lambda}} \left( 1 + O\left(u^{-1}\right) \right)
\end{equation}
uniformly in the range $1 \le m \le n$, where \begin{equation}\label{eq:defQ}
	D(x) = \exp\left( \sum_{j=1}^{m} \frac{x^j}{j}\right)x^{-n}
\end{equation} and $\lambda = \lambda(x) = \sum_{j=1}^{m} jx^j$.  Here $x$ is as defined in \eqref{eq:x def}. This result provides an asymptotic as long as $u \to \infty$.
\subsubsection{Inequalities}
Warlimont \cite{warlimont1991} proved the upper bound $\PP( f_n \text{ is }m\text{-smooth}) \le C\exp(-cu)$ in $1 \le m \le n$. Several lower bounds have been proved. Lovorn Bender and Pomerance \cite[Thm.~2.1]{lovorn1998} proved that $\PP( f_n \text{ is }m\text{-smooth})  \ge n^{-u}$ for $m\le \sqrt{n}$. Joux and Lercier proved $\log \PP(f_n \text{ is }m\text{-smooth}) \ge -(1+o_m(1)) u\log u$ for fixed $m$ and growing $q$ and $n$ \cite[App.~A]{joux2006}. Granville, Harper and Soundararajan \cite[Ex.~6]{Granville2015} show that $\PP( f_n \text{ is }m\text{-smooth}) \ge \rho(n/m)$ and state, without proof, $\PP( f_n \text{ is }m\text{-smooth}) \ge \rho(n/m)\exp(cn/m^2)$. This implies the Dickman function is not a good approximation once $m/\sqrt{n}$ is bounded from above. In the appendix we make this optimal and show that the Dickman function is not an asymptotic approximation if $m/\sqrt{n \log n}$ is bounded.
\subsubsection{Soundararajan's polynomial results and Ford's permutation results}
In an unpublished manuscript\footnote{Soundararajan's work is surveyed in \cite{Odlyzko1993,granville2008}. In \cite{Odlyzko2000,Schirokauer2002} it is explained that Soundararajan's work remained unpublished in view of the stronger result published in \cite{panario1998}; however, as shown in \S\ref{sec:inacc}, that work is flawed.} Soundararajan \cite[Thm.~1.1]{soundararajan}, building on \cite{manstavicius1992,manstavicius19922}, proved that
\begin{equation}\label{eq:sound1}
	\PP\left( f_n \text{ is }m\text{-smooth}\right) = \rho(u) \exp\left(O\left(\frac{n\log n}{m^2} \right)\right)
\end{equation}
uniformly for $n\ge m \ge \log(n \log^2 n)/\log q$. For $m \le \log_q n$ he gives lower and upper bounds  which are of different nature from \eqref{eq:sound1}. Additionally, he obtains an asymptotic formula for $\log \psi_q(n,m)$ uniformly in the full range $1\le m \le n$, with a relative error term $1/m + 1/\log n$. 

Very recently, Ford proved for $n \ge m \ge 1$ the upper bound $\PP(\pi_n \text{ is }m\text{-smooth}) \le e^{-u\log u+u-1}$ \cite[Thm.~1.16]{ford2021cycle} by elementary means. Additionally, he gave a short proof for for the following estimate \cite[Thm.~1.17]{ford2021cycle}
\begin{equation}\label{eq:fordsand}
	\rho\left(\frac{n}{m} \right) \le \PP(\pi_n \text{ is }m\text{-smooth}) \le \rho\left(\frac{n+1}{m+1} \right),
\end{equation}
holding for $n \ge m \ge 1$. It implies $\PP(\pi_n \text{ is }m\text{-smooth}) \sim \rho(u)$ for $m/\sqrt{n \log n} \to \infty$  \cite[Cor.~1.18]{ford2021cycle}. 

In \S\ref{sec:sound} we give a quick proof that Ford's \eqref{eq:fordsand} implies the following.\footnote{A slightly weaker version of Proposition~\ref{prop:perm} can be derived by plugging \eqref{eq:sound1} in Proposition~\ref{prop:bndd unif} and letting $q \to \infty$.}
\begin{prop}\label{prop:perm}
	For $n \ge m \ge 1$ we have 
	\[\PP(\pi_n \text{ is }m\text{-smooth}) = \rho(u)\exp\left(O\left( \frac{u\log (u+1)}{m}\right) \right).\]
\end{prop}
This extends \eqref{eq:pn transform} to the full range $n \ge m \ge 1$. Although Proposition~\ref{prop:perm} does not give an asymptotic result for $\PP(\pi_n \text{ is }m\text{-smooth})$ itself if $m$ is relatively small, it does show that $\log \PP(\pi_n \text{ is }m\text{-smooth}) \sim \log \rho(u)$ as $n,m \to \infty$. This behavior also holds for bounded $m$ by \cite[Thm.~1]{manstavicius2016}. We record this as
\begin{cor}
	As $n \to \infty$ we have $\log \PP(\pi_n \text{ is }m\text{-smooth}) \sim \log \rho(u)$, uniformly in $1\le m\le n$.
\end{cor}
Our methods allow us to deduce \eqref{eq:sound1} from Proposition~\ref{prop:perm}. Namely, we prove
\begin{thm}\label{thm:sound}
	Suppose $n \ge m \ge \log (n \log n)/\log q$. Then
	\[ \PP(f_n \text{ is }m\text{-smooth}) = \PP(\pi_n \text{ is }m\text{-smooth})  \exp \left( O\left( \frac{u \log (u+1)}{m}\right) \right).\]
\end{thm}
From Theorem~\ref{thm:sound} and Proposition~\ref{prop:perm} we immediately obtain \eqref{eq:sound1}.

\subsubsection{Some inaccuracies}\label{sec:inacc}
Let $\Psi(x,y):=\#\{ 1 \le n \le x: n\text{ is }y\text{-smooth}\}$. The Buchstab--de Bruijn identity states \cite[Eq.~(3.10)]{granville2008}
\begin{equation}\label{eq:bdidentity}
	\Psi(x,y) = 1+\sum_{p \le y} \Psi\left( \frac{x}{p},p\right),
\end{equation}
where the sum is over primes up to $y$.  De Bruijn used it to prove that \cite{debruijn1951}
\begin{equation}
	\Psi(x,y) = x \rho(u) \left( 1 + O_{\varepsilon}\left(\frac{\log(\frac{\log x}{\log y}+1)}{\log y}\right)\right)
\end{equation}
holds in the range $x \ge y \ge \exp((\log x)^{5/8+\varepsilon})$. We are not aware of a polynomial analogue of this identity. In the survey \cite{granville2008}, the identity
\begin{equation}
	\psi_q(n,m) - \psi_q(n,m-1)=\pi_q(m) \psi_q(n-m,m)
\end{equation}
is suggested as an analogue of \eqref{eq:bdidentity}, where $\pi_q(m)$ is the number of monic irreducibles of degree $m$. However, this identity is false already for $n=4$, $m=2$ and $q=3$.

In \cite{hildebrand1986}, Hildebrand extended de Bruijn's result to the range $x \ge y \ge \exp((\log \log x)^{5/3+\varepsilon})$, by using the following identity:
\begin{equation}
	\Psi(x,y)\log x=\int_{1}^{x} \frac{\Psi(t,y)}{t}dt +\sum_{\substack{p^m \le x \\ p\le y}} \Psi\left( \frac{x}{p^m},y\right)\log p,
\end{equation}
where the sum is over $y$-smooth prime powers up to $x$.
Hildebrand's identity does have a simple polynomial analogue, namely
\begin{equation}
	\psi_q(n,m) n = \sum_{\substack{\deg (P^k) \le n \\ \deg(P) \le m}} \psi_q\left(n-\deg(P^k),m\right) \deg P.
\end{equation}
It is proved in complete analogy with Hildebrand's original identity.

In \cite[Thm.~2]{manstavicius1992} (cf.~\cite[Thm.~A]{knopfmacher1997}) it is claimed that
\begin{equation}\label{eq:qmsave}
	\PP(f_n \text{ is }m\text{-smooth}) = \PP(\pi_n \text{ is }m\text{-smooth}) \left( 1 + O_{\varepsilon}(q^{-m(1/2-\varepsilon)})\right)
\end{equation}
holds uniformly in $m$ and $q$. This cannot hold as stated for small $m$, per the discussion in the introduction. In particular, a short computation shows that for $m = 1$, $\PP(f_n \text{ is }1\text{-smooth})/\PP(\pi _n \text{ is }1\text{-smooth}) \ge cn^2/q$, which contradicts \eqref{eq:qmsave} if, say, $n \ge q$. By Proposition \ref{prop:bndd unif}, \eqref{eq:qmsave} is true if one replaces $O_{\varepsilon}(q^{-m(1/2-\varepsilon)})$ with $O_n(q^{-m/2})$.

In \cite{panario1998} an estimate similar to Warlimont's estimate \eqref{eq:warlimont} is proved, but with an additional factor of $\log n$ in the numerator\footnote{The result is stated as $\PP(f_n \text{ is }m\text{-smooth})=\rho(u) (1+O(\frac{\log n}{m}))$, but -- as observed in Tenenbaum's review \cite{tenenbaumreview} -- this should read $\PP(f_n \text{ is }m\text{-smooth})=\rho(u) +O(\frac{\log n}{m})$.}. Moreover, a proof is sketched of the following estimate, for every integer $k \ge 2$:
\begin{equation}\label{eq:flajolet}
	\PP\left( f_n \text{ is }m\text{-smooth}\right) = \rho(u) + O_k\left( \frac{\log n}{m^k} \right)
\end{equation}
for $m <n/k$, as long as $m^k/\log n \to \infty$. However, this cannot hold as stated, even for bounded $u$ and $k=2$. Indeed, this violates the lower bound $\PP( f_n \text{ is }m\text{-smooth})  \ge \rho(u) + c\rho(u)u\log u/m$ (valid for $n/2 \ge m$) proven in the appendix.
\section{Strategy and preliminaries}
Throughout the paper, the letters $C$ and $c$ will stand for positive absolute constants that may change from one occurrence to the next.
\subsection{Strategy}\label{sec:strat}
We introduce the generating functions
\begin{equation}\label{eq:FFqdef}
	\begin{split}
		F(z)&:= 1+ \sum_{n \ge 1} \PP( \pi_n \text{ is }m\text{-smooth})z^n = 1+ \sum_{n \ge 1} \frac{\psi_{\pi}(n,m)}{n!}z^n,\\
		F_q(z)&:= 1+ \sum_{n \ge 1} \PP( f_n \text{ is }m\text{-smooth})z^n = 1+ \sum_{n \ge 1} \frac{\psi_q(n,m)}{q^n}z^n
	\end{split}
\end{equation}
whose analytic properties are explored in \S\ref{sec:gen} below. Throughout we use the notation
\[G_q(z):=F_q(z)/F(z).\]
We apply Cauchy's formula to $F$ and $F_q$:
\[ \PP(\pi_n \text{ is }m\text{-smooth}) = \frac{1}{2\pi i} \int_{|z|=r} \frac{F(z)}{z^{n+1}} dz, \qquad 
\PP(f_n \text{ is }m\text{-smooth}) = \frac{1}{2\pi i} \int_{|z|=r} \frac{F_q(z)}{z^{n+1}} dz.\]
One approach to studying these integrals is showing that the integrands are `close' to $\hat{\rho}(s)e^{us}$ after a change of variables, where $\hat{\rho}$ is the Laplace transform of $\rho$. This works when $u$ is `small' and implies $\PP(f_n \text{ is }m\text{-smooth})$ and $\PP(\pi_n \text{ is }m\text{-smooth})$ are asymptotic to $\rho(u)$ in some range \cite{manstavicius1992,manstavicius19922,manstavicius2016}. We modify this strategy: instead of studying the polynomial and permutations probabilities individually, we study the difference
\[ \PP(f_n \text{ is }m\text{-smooth})-\PP(\pi_n \text{ is }m\text{-smooth}) = \frac{1}{2\pi i} \int_{|z|=r} \frac{F_q(z)-F(z)}{z^{n+1}} dz.\]
We still think of the integrand as being related to $\hat{\rho}$ (which influences our choice of $r$), but now our aim is to upper bound the integral which is easier than studying the probabilities individually. Ultimately, this works because the ratio $G_q$ is close to $1$ in an appropriate sense when $u$ is `small'.

A second approach to the study of the integrals involves approximating the integrands $F(z)/z^{n+1}$ and $F_q(z)/z^{n+1}$ as gaussians by choosing the radii $r$ to be $x=x_{n,m}$ and $x_q=x_{q,n,m}$, respectively, where $x$ is the saddle point with respect to $F(z)/z^{n+1}$ (defined as the real positive solution to $-z(\log F(z))' = n$) and similarly $x_q$ is defined as the real solution to $-z(\log F_q(z))' = n$ \cite{manstavicius1992,manstavicius19922,manstavicius2016}. The saddle point $x$ coincides with $x$ defined in \eqref{eq:x def}.
Again, we modify this approach: we study both $F_q(z)/z^{n+1}$ and $F(z)/z^{n+1}$ near the saddle point $x$ associated with $F(z)/z^{n+1}$. We consider the difference
\[  \PP(f_n \text{ is }m\text{-smooth}) - G_q(x) \PP(\pi_n \text{ is }m\text{-smooth})= \frac{1}{2\pi i} \int_{|z|=x} \frac{F(z) (G_q(z)-G_q(x))}{z^{n+1}} dz.\]
We want to bound this integral, which leads to the study of $G_q$ and its derivatives.

Our strategy shares similarities with the work of Saias \cite{Saias1989}, who pioneered the \textit{indirect} saddle method. He studied $\Psi(x,y)$ by comparing it to de Bruijn's approximation $\Lambda(x,y)$.
\subsection{Primes}
We denote by $\pi_q(n):=|\mathcal{P} \cap \Mnq|$ the number of prime polynomials of degree $n$. From Gauss's identity \cite[Eq.~(1.3)]{arratia1993}
\begin{equation}\label{eq:gauss}
	\sum_{d \mid n} d \pi_q(d) =q^n,
\end{equation}
we obtain the estimate
\begin{equation}\label{eq:gaussbnd}
	\frac{q^n}{2n} \le \pi_q(n) \le \frac{q^n}{n}
\end{equation}
(cf.~\cite[Lem.~4]{pollack2013}) as well as \eqref{eq:ppt}.
\subsection{Generating functions}\label{sec:gen}
Since $\PP (\pi_n \text{ is }m\text{-smooth})$ and $\PP (f_n \text{ is }m\text{-smooth})$ are between $0$ and $1$, the generating series $F$ and $F_q$ defined in \eqref{eq:FFqdef} converge absolutely in  $|z|<1$ and define analytic function in the open disc. We shall show that they can be analytically continued to a larger region. The logarithm function will always be used with its principal branch.
\begin{lem}\label{lem:generating}
	We have
	\begin{equation}\label{eq:F form}
		F(z) = \exp\left( \sum_{i=1}^{m} \frac{z^i}{i} \right),
	\end{equation}
	while for every prime power $q$ we have 
	\begin{equation}\label{eq:gqai}
		G_q(z) =  \frac{F_q(z)}{F(z)}= \exp\left( \sum_{i>m} \frac{a_i}{i} z^i\right), \qquad a_i = a_{i,m,q} := q^{-i}\sum_{d \mid i, \, d \le m} d\pi_q(d).
	\end{equation}
	The coefficients $a_i$ satisfy, for all $i>m$, 
	\begin{equation}\label{eq:ai size}
		\frac{1}{2}q^{\max
			\{d\le m: \, d\textrm{ divides }i\}-i} \le a_i \le 2q^{\max
			\{d\le m: \, d\textrm{ divides }i\}-i} \le 2q^{\min\{m,\lfloor i/2 \rfloor\}-i} .
	\end{equation}
	In particular, the functions $F_q$ and $G_q$ are analytic in $|z|<q$, and $a_i=\Theta(q^{-i/2})$ for even $i \in [m+1,2m]$.
\end{lem}
\begin{proof}
	The exponential formula for permutations \cite[Cor.~5.1.9]{stanley1999enumerative} states the following. Given a function $g\colon \NN \to \CC$, we construct a corresponding function on permutations (on an arbitrary number of elements) as follows:
	\begin{equation}
		G(\pi) = \prod_{C \in \pi} g(|C|),
	\end{equation}
	where the product is over the disjoint cycles of $\pi$. We then have the following identity of formal power series:
	\begin{equation}
		1+\sum_{i \ge 1} \EE_{\pi \in S_i}  G(\pi) z^i = \exp\left(\sum_{j \ge 1} \frac{g(j)}{j}z^j\right).
	\end{equation}
	Applying the identity with $g(j) = \mathbf{1}_{j \le m}$, we obtain \eqref{eq:F form}. For $F_q$ we have, by unique factorization in $\FF_q[T]$,
	\begin{equation}
		F_q(z) = \prod_{\substack{P \in \mathcal{P}\\\deg(P) \le m}} \left( \sum_{i \ge 0} \left(\frac{z}{q}\right)^{\deg(P^i)}\right) =  \prod_{\substack{P \in \mathcal{P}\\\deg(P) \le m}} \left(1-\left(\frac{z}{q}\right)^{\deg(P)}\right)^{-1} = G_q(z)F(z)
	\end{equation}
	and
	\begin{multline}
		\log G_q(u) = \sum_{j \ge 1} \sum_{\deg(P) \le m} \frac{1}{j}\left(\frac{z}{q}\right)^{\deg(P)j} - \sum_{i=1}^{m} \frac{z^i}{i}\\ = \sum_{i \le m} \frac{z^i}{i} \bigg(q^{-i} \sum_{\substack{\deg(P)\le m\\\deg(P) \mid i}}\deg(P) - 1\bigg) + \sum_{i > m} \frac{z^i}{i} q^{-i} \sum_{\substack{\deg(P)\le m\\ \deg(P) \mid i}}\deg(P). 
	\end{multline}
	For $i \le m$ we have $\sum_{\deg(P)\le m,\, \deg(P) \mid i}\deg(P)/q^i = \sum_{d \mid i} d\pi_q(d)/q^i = 1$ by \eqref{eq:gauss}, proving \eqref{eq:gqai}. The bound \eqref{eq:ai size} now follows from \eqref{eq:gaussbnd}.
\end{proof}
\section{Proof of Proposition~\ref{prop:bndd unif}}
We write $[z^n]H(z)$ the for the $n$th coefficient in a power series $H$. By definition, $\PP(f_n \text{ is }m\text{-smooth}) = [z^n]F(z)$ and $\PP(\pi_n \text{ is }m\text{-smooth}) = [z^n]F_q(z)$, where $F$ and $F_q$ are defined in \S\ref{sec:gen}. We are set out to prove
\begin{equation}
	0 \le [z^n](F_q-F) \le \frac{C}{ mq^{\lceil \frac{m+1}{2}\rceil}}.
\end{equation}
By Lemma~\ref{lem:generating}, 
\begin{equation}
	F_q-F = F(G_q-1),
\end{equation}
and both $F$ and $G_q-1$ have non-negative coefficients. This proves $[u^n](F_q-F) \ge 0$. For the upper bound we also use the non-negativity, which implies that 
\begin{equation}
	[z^n](F_q-F) \le (\max_{0 \le i \le n} [z^i] F) (G_q(1)-1).
\end{equation}
We have 
\begin{equation}
	[z^i]F = \PP(\pi_i \text{ is }m\text{-smooth}) \le 1,
\end{equation}
and so $[z^n](F_q-F) \le G_q(1)-1$. By \eqref{eq:ai size},
\begin{equation}
	0 \le \sum_{i>m} \frac{a_i}{i} \le \frac{2}{m}\sum_{i>m}\frac{1}{q^{\lceil \frac{i}{2}\rceil}} \le \frac{C}{q^{\lceil \frac{m+1}{2}\rceil}m} \le C,
\end{equation}
so that
\begin{equation}\label{eq:gq1}
	G_q(1) \le \exp\left(\frac{C}{mq^{\lceil \frac{m+1}{2}\rceil}}\right) = 1+O\left( \frac{1}{mq^{\lceil \frac{m+1}{2} \rceil}}\right)
\end{equation}
and the required bound follows.\qed
\section{Analysis via Laplace transform}
Here we shall use properties of the Laplace transform of $\rho$  to deduce Theorem~\ref{thm:laplace} in a limited range.
\begin{thm}\label{thm:laplace indeed}
	If $n \ge m \ge C \sqrt{n \log n}$, then
	\begin{equation}\label{eq:main largen2}
		\frac{\PP(f_n \text{ is }m\text{-smooth})}{\PP(\pi_n \text{ is }m\text{-smooth})} =1+ O\left(\frac{u \log (u+1)}{ mq^{\lceil \frac{m+1}{2}\rceil}}\right).
	\end{equation}
\end{thm}
\subsection{Asymptotics of parameters}\label{sec:i t}
We define $\xi \colon [1,\infty) \to [0,\infty)$, a function of variable $u\ge 1$, by
\begin{equation}\label{eq:def xi}
	e^{\xi} = 1+u\xi.
\end{equation}
\begin{lem}\cite[Lem.~1]{hildebrand1984}\label{lem:xi size}
	We have $\xi \sim \log u$ as $u \to \infty$, and $\xi'=u^{-1}(1+O(1/\log u))$.
\end{lem}
Recall we have 
\begin{equation}
	H_n:=\sum_{i=1}^{n} \frac{1}{i} = \log n + \gamma + O\left( \frac{1}{n} \right)
\end{equation}
for the Euler--Mascheroni constant $\gamma$. Define the entire function
\begin{equation}\label{eq:I def}
	I(s)=\int_{0}^{s} \frac{e^v-1}{v}\, dv.
\end{equation}
Note that $I(\xi)$ grows faster than any polynomial in $\xi$.
\begin{lem}\cite[Thm.~2.1, Lem.~2.6]{hildebrand1993}\label{lem:rho i transform}
	We have
	\begin{equation}\label{eq:hat rho}
		\hat{\rho}(s) := \int_{0}^{\infty} e^{-sv}\rho(v)\, dv = \exp\left( \gamma + I(-s) \right)
	\end{equation}
	for all $s \in \CC$. Also,
	\begin{equation}\label{eq:rho and i}
		\exp\left( \gamma-u\xi + I(\xi)\right) = \rho(u)\sqrt{\frac{2\pi}{\xi'}}\left(1 + O\left( \frac{1}{u}\right)\right).
	\end{equation}
\end{lem}
\begin{lem}\cite[Lem.~2.7]{hildebrand1993}\label{lem:i bounds}
	The following bounds hold for $s=-\xi(u)+i\tau$, $\tau \in \RR$:
	\begin{equation}
		\hat{\rho}(s) = \begin{cases} O\left(\exp\left(I(\xi)-\frac{\tau^2u}{2\pi^2}\right)\right) & \mbox{if }|\tau| \le \pi,\\
			O\left(\exp\left(I(\xi)-\frac{u}{\pi^2+\xi^2}\right)\right) & \mbox{if }|\tau| \ge \pi,\\
			\frac{1}{s} + O\left( \frac{1+u \xi }{s^2} \right) & \mbox{if }|\tau| \ge 1+u\xi.\end{cases}
	\end{equation}
\end{lem}
We define a function $T(s)$ which arises in \eqref{eq:FandT} when relating the generating function $F(z)$ (at $z=e^{-s/m}$) to $\hat{\rho}(s)$:
\begin{equation}
	T(s)=\int_{0}^{s} \frac{e^v-1}{v}\left( \frac{v}{m} \frac{e^{v/m}}{e^{v/m}-1}-1\right)\, dv.
\end{equation}
It is analytic in the strip $|\Im s| < 2 \pi m$.
\begin{lem}\cite[Lem.~11]{manstavicius2016} \label{lem:t}
	Let $s=\eta+i\tau$, $0 \le \eta \le \pi m$ and $-\pi m \le \tau \le \pi m$. We have
	\begin{equation}
		\left|T(s)+\frac{s}{2m}\right| \ll \frac{e^{\eta}}{m} + \frac{\tau^2}{m^2}.
	\end{equation}
\end{lem}

\begin{lem}\label{lem:gq exi}
	Suppose $n \ge m \ge \sqrt{n \log n}$. If $n$ is sufficiently large we have
	\begin{equation}
		0 \le G_q(e^{\xi/m})-1 \le \frac{C u \log(u+1)}{mq^{\lceil \frac{m+1}{2} \rceil}}.
	\end{equation}
\end{lem}
\begin{proof}
	By Lemma~\ref{lem:xi size} we have $e^{\xi}=1+u\xi  \le 1+Cu\log u \le 1+Cn\log n\le 1+Cm^2 \le   q^{m/3}$ if $n$ is sufficiently large, and so $e^{\xi/m}/\sqrt{q} \le q^{-1/6} \le 2^{-1/6}$. We have
	\begin{equation}
		0 \le \sum_{i>m} \frac{a_i}{i} (e^{\xi/m})^i \ll \frac{1}{m} \sum_{i > m} (e^{\xi/m})^i q^{-\lceil i/2 \rceil} \ll \frac{1}{m} \left( \frac{e^{\xi(m+1)/m}}{q^{\lceil \frac{m+1}{2}\rceil}} + \frac{e^{\xi(m+2)/m}}{q^{\lceil \frac{m+2}{2}\rceil}} \right) \ll \frac{u \log(u+1)}{mq^{\lceil \frac{m+1}{2} \rceil}} \ll 1.
	\end{equation}
	Since $0 \le e^y-1\ll y$ for bounded non-negative $y$, the inequality follows.
\end{proof}
\begin{lem}\label{lem:parts}
	Let $A>0$. Suppose $m\pi \ge 1+u \xi$ and $u \ge \max\{A+1, 3\}$. Let $\Delta_2 =  \{ -\xi + i \tau: \tau \in [-m\pi,-(1+u\xi)]\cup[1+u\xi,m\pi]\}$. Then
	\begin{equation}
		\int_{\Delta_2} e^{(u-A)s} \hat{\rho}(s)\, ds \ll \exp(\xi A) \rho(u).
	\end{equation}
\end{lem}
\begin{proof}
	Without loss of generality, we restrict our attention to $\Delta'_{2} = 
	\{ -\xi + i\tau : \tau \in [1+u\xi,m\pi]\} \subset \Delta_2$, and the other part is bounded in the same way. We integrate by parts, differentiating $\hat{\rho}(s)$ using \eqref{eq:hat rho} and \eqref{eq:I def}:
	\begin{equation}
		\int_{\Delta'_2} e^{(u-A)s} \hat{\rho}(s)\, ds = \frac{e^{(u-A)s}}{u-A}\hat{\rho}(s) \Biggr|_{-\xi+i(1+u\xi)}^{-\xi+im\pi} - \frac{1}{u-A}\int_{\Delta'_2} e^{(u-A)s} \hat{\rho}(s) \frac{e^{-s}-1}{s} \, ds.
	\end{equation} 
	Since $|\exp((u-A)s)| \le \exp(-\xi(u-A))$ and $\int_{\Delta'_2} |ds|/|s|^2 = O(1)$, the result follows by Lemmas~\ref{lem:rho i transform}--\ref{lem:i bounds} and the triangle inequality. We exploit the fact that $\exp(-u\xi)\ll_C \rho(u)/u^C$ for any $C>0$, due to the factor $\exp(I(\xi))$ in \eqref{eq:rho and i}.
\end{proof}

\subsection{Proof of Theorem~\ref{thm:laplace indeed}}
Suppose $u \le M$. By \eqref{eq:pn transform} (or Theorem~\ref{thm:lower}) and Proposition~\ref{prop:bndd unif}, $\PP(f_n \text{ is }m\text{-smooth}) \ge c_M > 0$. Proposition~\ref{prop:bndd unif} now implies that 
\begin{equation}\label{eq:main largen u bnd}
	\PP(f_n \text{ is }m\text{-smooth}) = \PP(\pi_n \text{ is }m\text{-smooth}) \left(1 + O_M\left(\frac{1}{ mq^{\lceil \frac{m+1}{2}\rceil}}\right)\right),
\end{equation}
which establishes the theorem in the case of bounded $u$. Hence, we may assume $u \gg 1$ in our argument.

Let $y:=e^{\xi/m}$, where $\xi$ is defined in \eqref{eq:def xi}. By Proposition~\ref{prop:bndd unif},
\begin{equation}\label{eq:bndd}
	\frac{\PP(f_n \text{ is }m\text{-smooth})}{\PP(\pi_n \text{ is }m\text{-smooth})} = 1+ O_n\left(\frac{1}{q^{\lceil \frac{m+1}{2}\rceil}}\right)
\end{equation}
and so we may assume that $n$ is sufficiently large. In particular, by Lemma~\ref{lem:xi size}, $|y| \le e^{C\log u/m} \le e^{C\log n/\sqrt{n \log n}} \le 3/2 < q$ for sufficiently large $n$. Since $F$ and $F_q$ are analytic in $|z| < q$ by Lemma~\ref{lem:generating}, we have, by Cauchy's integral formula,
\begin{equation}
	\PP(\pi_n \text{ is }m\text{-smooth}) = \frac{1}{2\pi i} \int_{|z|=y} \frac{F(z)}{z^{n+1}}\, dz , \qquad \PP(f_n \text{ is }m\text{-smooth}) = \frac{1}{2\pi i} \int_{|z|=y} \frac{F_q(z)}{z^{n+1}}\, dz.
\end{equation}
Using the parametrization $z=e^{-s/m}$ with $s=-\xi-i\tau$, $-m\pi \le \tau \le m\pi$, we obtain
\begin{equation}
	\PP(\pi_n \text{ is }m\text{-smooth}) = \frac{1}{2\pi im} \int_{\Delta} e^{us} F(e^{-s/m}) \,ds, \qquad \PP(f_n \text{ is }m\text{-smooth}) = \frac{1}{2\pi im} \int_{\Delta} e^{us} F_q(e^{-s/m})\, ds
\end{equation}
where $\Delta:=\{-\xi+i\tau: -m\pi \le \tau \le m\pi\}$. By Lemma~\ref{lem:generating}, as in the proof of \cite[Cor.~3]{manstavicius2016},
\begin{multline}\label{eq:log f i t}
	\log F(z) - H_m = \sum_{i=1}^{m} \frac{z^i-1}{i} = \int_{1}^{z} \sum_{i=1}^{m} t^{i-1}\,dt =\int_{1}^{z} \frac{t^m-1}{t-1}\, dt  \\= \int_{0}^{m\log z} \frac{e^v-1}{v}\frac{v}{m} \frac{dv}{1-e^{-v/m}} = I(m\log z) + T(m\log z)
\end{multline}
where $I$, $T$ are defined in \S\ref{sec:i t} and $z \ge 1$. Hence
\begin{equation}\label{eq:f h t}
	F(e^{-s/m}) = \exp\big( H_m + I(-s) + T(-s)\big),
\end{equation}
which holds for all $s \in \Delta$ (not only $s\le 0$) by the uniqueness principle. By \eqref{eq:hat rho} and \eqref{eq:f h t},
\begin{equation}\label{eq:FandT}
	F(e^{-s/m}) = \exp\big(H_m -\gamma + T(-s)\big)\hat{\rho}(s).
\end{equation}
Hence,
\begin{equation}\label{eq:diff laplace}
	\PP(f_n \text{ is }m\text{-smooth})-\PP(\pi_n \text{ is }m\text{-smooth}) = \frac{\exp\big(H_m -\gamma\big)}{2\pi im} \int_{\Delta} e^{us} \hat{\rho}(s) \exp(T(-s)) (G_q(e^{-s/m})-1)\,ds.
\end{equation}
We have $|\exp(H_m -\gamma)/m| =O(1)$, and we turn to bound the integral in the right-hand side of \eqref{eq:diff laplace}. We partition $\Delta$ as $\Delta_0 \cup \Delta_1 \cup \Delta_2$, where $\Delta_0=\{-\xi+i\tau: -\pi \le \tau \le \pi\}$, $\Delta_1=\{-\xi+i\tau: \tau \in [-(1+u\xi),-\pi] \cup [\pi, 1+u\xi]\}$ and $\Delta_2 = \{-\xi+i\tau: \tau \in [-m\pi,-(1+u\xi)] \cup [1+u\xi,m\pi]\}$. We may assume that $m\pi > 1+u\xi > \pi$ since we can take $n\gg 1$ and $u \gg 1$. By Lemma~\ref{lem:t}, $T(-s)$ is bounded in $\Delta$, and so 
\begin{equation}
	\begin{split}
		\big| \int_{\Delta_i} e^{us} \hat{\rho}(s) \exp(T(-s)) (G_q(e^{-s/m})-1)\,ds\big| &\le  \int_{\Delta_i} \big| e^{us} \hat{\rho}(s) \exp(T(-s)) (G_q(e^{-s/m})-1)\,ds\big|\\
		& \ll \exp(-u\xi) (G_q(e^{\xi/m})-1) \int_{\Delta_i} |\hat{\rho}(s)| \, |ds|
	\end{split}
\end{equation}
for $i=0,1$. For $i=0$ we have by the first part of Lemma~\ref{lem:i bounds}
\begin{equation}
	\int_{\Delta_0} |\hat{\rho}(s)| \, |ds| \ll \exp\big( I(\xi)\big) \int_{-\pi}^{\pi} \exp(-\tau^2u/2\pi^2)\,d\tau \ll   \frac{\exp\big( I(\xi)\big)}{\sqrt{u}},
\end{equation}
and so, by Lemmas~\ref{lem:xi size} and \ref{lem:rho i transform},
\begin{equation}
	\big| \int_{\Delta_0} e^{us} \hat{\rho}(s) \exp(T(-s)) (G_q(e^{-s/m})-1)\,ds\big| \ll \rho(u)  (G_q(e^{\xi/m})-1).
\end{equation}
Similarly, for $i=1$ we have by the second part of Lemma~\ref{lem:i bounds}
\begin{equation}
	\int_{\Delta_1} |\hat{\rho}(s)| \, |ds| \ll \exp\big( I(\xi)\big) \exp(-u/(\pi^2+\xi^2)) (1+u\xi) ,
\end{equation}
and so, by Lemmas~\ref{lem:xi size} and \ref{lem:rho i transform},
\begin{equation}
	\begin{split}
		\big| \int_{\Delta_1} e^{us} \hat{\rho}(s) \exp(T(-s)) (G_q(e^{-s/m})-1)\,ds\big| &\ll \rho(u) \sqrt{u}\exp(-u/(\pi^2+\xi^2))(1+u\xi) (G_q(e^{\xi/m})-1)\\
		&\ll \rho(u)(G_q(e^{\xi/m})-1).
	\end{split}
\end{equation}
As $G_q(e^{\xi/m})-1 = O(u \log(u+1)/(mq^{\lceil (m+1)/2\rceil}))$ by Lemma~\ref{lem:gq exi}, the integrals over $\Delta_0$ and $\Delta_1$ contribute at most
\begin{equation}\label{eq:size integral}
	\ll \frac{\rho(u)u\log(u+1)}{mq^{\lceil (m+1)/2\rceil}}. 
\end{equation}
We wish to bound the integral over $\Delta_2$ by the same quantity. However, using the triangle inequality as before we will incur an extra factor of $\log m$ coming from the integral of $|\hat{\rho}(s)| = O(1/|s|)$ on $\Delta_2$. To remove this factor, we first replace $\exp(T(-s))$ by $1$ -- the error obtained is acceptable, since by Lemma~\ref{lem:t} 
\begin{align}
\int_{\Delta_2} e^{us} &\hat{\rho}(s) (e^{T(-s)}-1)(G_q(e^{-s/m})-1) ds \\
&\ll \int_{\Delta_2} |e^{us} \hat{\rho}(s) T(-s)| |(G_q(e^{-s/m})-1)| |ds|   \\
&\ll e^{-u\xi} (G_q(e^{\xi/m})-1)\int_{\Delta_2}  |s|^{-1} \left(\left|\frac{s}{m}\right| + \left|\frac{s}{m}\right|^2+\frac{u\log(u+1)}{m}\right) |ds|
\end{align}
and this is $\ll \rho(u)(G_q(e^{\xi/m})-1)$ which we saw is at most \eqref{eq:size integral}. It remains to bound
\begin{equation}\label{eq:delta2 int}
	\int_{\Delta_2} e^{us} \hat{\rho}(s) (G_q(e^{-s/m})-1)\,ds.
\end{equation}
Lemma~\ref{lem:generating} shows we may write
\begin{equation}
	G_q(e^{-s/m})-1 = \sum_{i=m+1}^{4m} \frac{a_i \left(e^{-s/m}\right)^{i}}{i} + O\left( \frac{e^{2\xi}}{q^{m}}\right)
\end{equation}
with $a_i=O(q^{-\lceil i/2 \rceil})$. The integral $\int_{\Delta_2} |e^{us} \hat{\rho}(s) e^{2\xi}/q^{m} |ds$ is sufficiently small (smaller than \eqref{eq:size integral}). The term corresponding to $i$ is bounded as follows by Lemma~\ref{lem:parts}:
\begin{equation}
	\int_{\Delta_2} e^{us} \hat{\rho}(s)\frac{a_i \left(e^{-s/m}\right)^{i}}{i}\, ds \ll \exp\left( \xi i/m\right) \frac{\rho(u)}{mq^{\lceil \frac{i}{2} \rceil}}.
\end{equation}
Summing over $m+1 \le i \le 4m$, we see that \eqref{eq:delta2 int} is smaller than \eqref{eq:size integral}. All in all,
\begin{equation}
	\PP(f_n \text{ is }m\text{-smooth}) - \PP(\pi_n \text{ is }m\text{-smooth}) = O\left( \rho(u) \frac{u\log(u+1)}{ q^{\lceil \frac{m+1}{2}\rceil}m}\right).
\end{equation}
By \eqref{eq:pn transform} (or Theorem~\ref{thm:lower}), $\rho(u) \ll \PP(\pi_n \text{ is }m\text{-smooth})$, which gives the desired result. \qed

\section{Saddle point analysis}
We shall deduce Theorems~\ref{thm:saddle} and \ref{thm:saddle2} from the following
\begin{thm}\label{thm:saddle3}
	If $\min\{n/(\log n \log^3\log  (n+1)), \, n/3\} \ge m > \log_q n$ then
	\begin{equation}\label{eq:main mediumn4}
		\frac{\PP(f_n \text{ is }m\text{-smooth})}{\PP(\pi_n \text{ is }m\text{-smooth})} =  G_q(x)\left(1 + O\left( \frac{G_q''(x)x^2+G_q'(x)xm}{nmG_q(x)}\right)\right).
	\end{equation}
\end{thm}
\subsection{Asymptotics of parameters}
Recall that $x$ is the positive constant defined by $\sum_{j=1}^{m} x^j = n$, and that $\lambda=\sum_{j=1}^{m} jx^j$.
\begin{lem}\cite[Lem.~9]{manstavicius2016} \label{lem:lambda}
	For $u>1$ we have $|\lambda - mn| \le mn/\log u$. For $u \ge 3$ we have
	\begin{equation}\label{eq:x size}
		x^m = \Theta\left(n \min\left\{1, \frac{\log u}{m}\right\}\right).
	\end{equation}
\end{lem}

\begin{lem}\label{lem:deriv}
	Suppose $n/3 \ge m \ge (2+\varepsilon)\log_q n$ for some $\varepsilon>0$. Let $\parity = \mathbf{1}_{2 \mid m}$. Then 
	\begin{equation}\label{eq:gqx}
		1+\frac{cux^{1+\parity}}{q^{\lceil \frac{m+1}{2} \rceil}} \min\left\{1, \frac{\log u}{m}\right\} \le G_q(x) \le 1+ \frac{C_{\varepsilon}ux^{1+\parity}}{q^{\lceil \frac{m+1}{2} \rceil}} \min\left\{1, \frac{\log u}{m}\right\},
	\end{equation}
	and in particular, 
	\begin{equation}\label{eq:gqx2}
		G_q(x) = 1+\Theta_{n,\varepsilon}\left( \frac{1}{q^{\lceil \frac{m+1}{2}\rceil}}\right).
	\end{equation}
	Moreover,
	\begin{equation}
		|G'_q(x)| \le C_{\varepsilon} \frac{nx^{\parity}}{q^{\lceil \frac{m+1}{2} \rceil}}\min\left\{1,\frac{\log u}{m}\right\},\qquad |G''_q(x)| \le C_{\varepsilon} \frac{nmx^{\parity-1}}{q^{\lceil \frac{m+1}{2} \rceil}} \min\left\{1, \frac{\log u}{m}\right\}.
	\end{equation}
\end{lem}
\begin{proof}
	By \eqref{eq:ai size},
	\begin{equation}\label{eq:ai aux}
		\sum_{i>m} \frac{a_i}{i}x^i \le \sum_{i>m}\frac{2x^i}{mq^{\lceil \frac{i}{2}\rceil}}.
	\end{equation}
	Since $a_i \gg q^{-i/2}$ if $i\in [m+1,2m]$ is even, we also have
	\begin{equation}\label{eq:ai aux lower}
		G_q(x)-1 \ge \sum_{i>m} \frac{a_i}{i}x^i \ge \frac{cx^m}{m}(a_{m+1}x+a_{m+2}x^2) \ge \frac{cx^m x^{1+\parity}}{mq^{\lceil \frac{m+1}{2}\rceil}} .
	\end{equation}
	As $1 \le x \le n^{1/m} \le q^{1/(2+\varepsilon)}$, the right-hand side of \eqref{eq:ai aux} is at most  $C_{\varepsilon}(x^m/m)x^{1+\parity}q^{-\lceil (m+1)/2\rceil}$ (consider even and odd $i$ separately). As $x \le q^{1/2}$, this is $O_{\varepsilon}(1)$, which also proves the upper bound for $G_q(x)$ in \eqref{eq:gqx}. For the lower bound, \eqref{eq:x size} guarantees that the right-hand side of \eqref{eq:ai aux lower} is at least as stated.
	
	As $C_\varepsilon ux^{1+\parity} =O_{n,\varepsilon}(1)$, we obtain \eqref{eq:gqx2}. Similarly, as  $x^m \le n$,
	\begin{equation}
		\sum_{i>m} a_i x^{i-1} \le C_\varepsilon \left( \frac{x^m}{q^{\lceil \frac{m+1}{2} \rceil}} + \frac{x^{m+1}}{q^{\lceil \frac{m+2}{2}\rceil }} \right) \le C_\varepsilon \frac{nx^{\parity}}{q^{\lceil \frac{m+1}{2} \rceil}}\min\left\{1,\frac{\log u}{m}\right\} \le \frac{C_{\varepsilon}}{x},
	\end{equation}
	yielding the bound on $G'_q(x)$ since we have the identity $G'_q(x) = G_q(x) \sum_{i >m} a_i x^{i-1} $ and the above shows $G_q(x) =O_{\varepsilon} (1)$. Considering separately $m<i \le 2m+1$ and $i \ge 2m+2$, 
	\begin{equation}
		\sum_{i>m} (i-1)a_i x^{i-2} \le C_\varepsilon m\left( \frac{x^{m-1}}{q^{\lceil \frac{m+1}{2} \rceil}} + \frac{x^m}{q^{\lceil \frac{m+2}{2} \rceil}}+ \frac{x^{2m}}{q^{m+2}} \right) \le C_\varepsilon \frac{nmx^{\parity-1}}{q^{\lceil \frac{m+1}{2} \rceil}} \min\left\{1, \frac{\log u}{m}\right\},
	\end{equation}
	which yields the desired bound on $G''_q(x)$ as $G''_q(x) = G_q(x) ( (\sum_{i>m} a_i x^{i-1})^2 + \sum_{i>m} (i-1)a_i x^{i-2})$.
\end{proof}

\begin{lem}\label{lem:small m}
	Fix $\varepsilon>0$. Suppose $(3-\varepsilon) \log_q n \ge m \ge (1+\varepsilon)\log_q n$. Letting
	\begin{equation}\label{eq:Ssize}
		S  = \frac{n}{q^{m/2}} \left(\frac{n^{1/m}}{\sqrt{q}}\right)^{1+\parity}\sum_{j=0}^{\lceil \frac{m-2}{2}\rceil} \left( \frac{n^{2/m}}{q}\right)^j
	\end{equation}
	where $\parity = \mathbf{1}_{2 \mid m}$, we have
	\begin{equation}
		\frac{G_q'(x)x}{G_q(x)} =\Theta_{\varepsilon}\left( S + \frac{n^2}{q^m}\right), \qquad 
		\frac{G_q''(x)x^2}{G_q(x)} =\Theta_{\varepsilon}\left( \left( S + \frac{n^2}{q^{m}}\right) \left( S + \frac{n^2}{q^{m}}+m\right)\right),
	\end{equation}
	\[\log G_q(x) =\Theta_{\varepsilon}\left( \frac{S}{m} + \frac{n^2}{m q^m }\right).\]
	Moreover,
	\begin{equation}
		S \asymp \begin{cases} m & \text{if }|m-2\log_q n|\le  1/\log q,\\ 
			\left(1-\frac{n^{2/m}}{q}\right)^{-1}		\left(\frac{n}{q^{m/2}}\right)^{1+\frac{1+a}{m}} & \text{if }m-2\log_q n \ge 1/\log q,\\ 
			\left(1-\frac{q}{n^{2/m}}\right)^{-1} \frac{n^2}{q^m} & \text{if }m-2\log_q n \le -1/\log q.
		\end{cases}
	\end{equation}
\end{lem}
\begin{proof}
	Letting
	\[ S'= \sum_{i=m+1}^{2m} a_ix^i\]
	where $a_i$ are defined in \eqref{eq:gqai}, we have
	\[ S' \asymp S\]
	by \eqref{eq:ai size} and \eqref{eq:x size}. To prove the estimate for $G'_q(x)x/G_q(x)$ we use \eqref{eq:ai size} and argue as follows:
	\begin{equation}
		\frac{G_q'(x)x}{G_q(x)} = \sum_{i>m} a_i x^i \asymp S' + \sum_{i\ge 2m} a_i x^i =  \Theta_{\varepsilon}\left(S'+\frac{n^2}{q^m}\right).
	\end{equation}
	The estimate for $G_q''(x)x^2/G_q(x)$ follows similarly from 
	\begin{equation}
		\frac{G_q''(x)x^2}{G_q(x)} = \left( \sum_{i>m} a_i x^i \right)^2 + \sum_{i>m} a_ix^i(i-1) =\Theta_{\varepsilon}\left( \left(\frac{G'_q(x)x}{G_q(x)}\right)^2+mS'+\frac{mn^2}{q^m}\right).
	\end{equation}
	In the same way, the estimate for $\log G_q(x)$ follows from
	\begin{equation}
		\log G_q(x) = \sum_{i=m+1}^{2m-1} \frac{a_ix^i}{i} + \sum_{i\ge 2m} \frac{a_ix^i}{i} = \Theta_{\varepsilon}\left(\frac{S'}{m} + \frac{n^2}{m q^m }\right).
	\end{equation}
	We estimate $S$ by the following general estimate for geometric sums: for $a>0$ we have that $\sum_{i=0}^{d-1} a^i$ is $\Theta(d)$ if $|a-1|\le c/d$, is $\Theta(a^d/(a-1))$ if $a-1 \ge c/d$ and is $\Theta(1/(1-a))$ if $a-1 \le -c/d$.
\end{proof}

\subsection{Proof of Theorem~\ref{thm:saddle3}}
Since $F$ is entire, we have, by Cauchy's integral formula,
\begin{equation}
	\PP(\pi_n \text{ is }m\text{-smooth}) = \frac{1}{2\pi i} \int_{|z| = x} \frac{F(z)}{z^{n+1}}\, dz.
\end{equation}
The parametrization $z=xe^{it}$, $t \in [-\pi,\pi]$ shows, by Lemma~\ref{lem:generating}, that
\begin{multline}
	\PP(\pi_n \text{ is }m\text{-smooth}) = \frac{1}{2\pi} \int_{-\pi}^{\pi} \exp\left( \sum_{j=1}^{m} \frac{x^j e^{itj}}{j} - n(it+\log x)\right) \, dt \\= \frac{D(x)}{2\pi} \int_{-\pi}^{\pi} \exp\left( \sum_{j=1}^{m} \frac{x^j (e^{itj}-1)}{j}-int \right) \, dt,
\end{multline}
where $D(x)$ is defined in \eqref{eq:defQ}.
As $F_q$ is analytic in $|z|<q$, and $x < q$, we similarly have, by Lemma~\ref{lem:generating}, that
\begin{equation}
	\PP(f_n \text{ is }m\text{-smooth}) =\frac{1}{2\pi i} \int_{|z| = x} \frac{F_q(z)}{z^{n+1}}\, dz= \frac{D(x)}{2\pi} \int_{-\pi}^{\pi} \exp\left( \sum_{j=1}^{m} \frac{x^j (e^{itj}-1)}{j} - int\right) G_q(xe^{it})\, dt.
\end{equation}
Hence, 
\begin{multline}\label{eq:diff p}
	\PP(f_n \text{ is }m\text{-smooth}) - \PP(\pi_n \text{ is }m\text{-smooth}) G_q(x)  \\=  \frac{D(x)}{2\pi} \int_{-\pi}^{\pi} \exp\left( \sum_{j=1}^{m} \frac{x^j (e^{itj} - 1)}{j}-int\right)  (G_q(xe^{it}) - G_q(x))\, dt.
\end{multline}
By assumption, $u \ge 3$. Let $t_0=n^{-1/3}m^{-2/3}$ and $t_1 = 1/m$. We shall bound separately the contribution of $|t|\le t_0$, $t_0 \le |t|\le t_1$ and $t_1 \le |t| \le \pi$ to the right-hand side of \eqref{eq:diff p}. For $k=1,2,3$, let
\begin{equation}
	I_k = \int_{A_k}  \exp\left( \sum_{j=1}^{m} \frac{x^j (e^{itj} - 1)}{j}-int\right)  \left(G_q(xe^{it}) - G_q(x)\right)\, dt,
\end{equation}
where  $A_1 = [-t_0,t_0]$, $A_2 = [-t_1,t_1]\setminus A_1$ and $A_3 = [-\pi,\pi]\setminus A_2$. A second-order Taylor approximation shows that for $|t| \le \pi$,
\begin{equation}\label{eq:taylor g}
	\begin{split}
		G_q(xe^{it})-G_q(x) &= G'_q(x) (xe^{it}-x) + O\left(G''_q(x)\left|xe^{it}-x\right|^2\right) \\
		&= iG'_q(x)xt  + O\left(\left(G''_q(x)x^2 + G'_q(x)x\right)t^2\right).
	\end{split}
\end{equation}
For $|t| =O(1/m)$ and $ 1\le j \le m$ we have $e^{itj}-1 = itj -(tj)^2/2 + O(|t|^3j^3)$, so that
\begin{equation}\label{eq:taylor sum}
	\sum_{j=1}^{m} \frac{x^j (e^{itj} - 1)}{j}-int = it\left( \sum_{j=1}^{m} x^j - n\right) -\frac{t^2}{2} \sum_{j=1}^{m} x^j j + O\left(|t|^3 \sum_{j=1}^{m} x^j j^2\right) = - \frac{\lambda t^2}{2} + O\left( |t|^3 \lambda_2\right)
\end{equation}
where 
\begin{equation}
	\lambda_2 := \sum_{j=1}^{m} x^j j^2 \le m \sum_{j=1}^{m} x^j j = m \lambda.
\end{equation}
In $A_1$ we have $|t|^3\lambda_2 = O(n^{-1}m^{-2}m\lambda)=O(1)$ by Lemma~\ref{lem:lambda}. By \eqref{eq:taylor g} and \eqref{eq:taylor sum},
\begin{equation}
	I_1 =  \int_{-t_0}^{t_0} \exp\left(-\frac{\lambda t^2}{2}\right) \left(iG'_q(x)xt + O\left(|t|^2 a + |t|^4 b\lambda_2 + |t|^5 a\lambda_2\right)\right)\, dt
\end{equation}
for $a = G_q''(x)x^2+G_q'(x)x$, $b=G'_q(x)x$. We have $\int_{-t_0}^{t_0} \exp(-\lambda t^2/2) t \, dt= 0$ and $\int_{\RR} \exp(-\lambda t^2/2) |t|^k\, dt \ll  \lambda^{-(k+1)/2}$ for $k=2,4,5$. Hence
\begin{equation}
	|I_1| \ll \lambda^{-3/2} a + \lambda^{-5/2}b\lambda_2 + \lambda^{-3} a \lambda_2 \ll \frac{1}{\lambda^{3/2}} \left( a+bm  + \frac{am}{\sqrt{\lambda}}\right).
\end{equation}
As $n \ge 3m$, we have $\lambda \ge cnm$ by Lemma~\ref{lem:lambda}. Thus
\begin{equation}\label{eq:i1 bnd}
	|I_1| \ll  \frac{1}{\sqrt{\lambda}}\frac{a+bm}{nm}.
\end{equation}
We turn our attention to $I_2$. We have $1-\cos s \ge cs^2$ in $|s| \le \pi$, and so 
\begin{equation}
	\left|\exp\left( \sum_{j=1}^{m} \frac{x^j (e^{itj} - 1)}{j}-int\right)\right| \le \exp\left( -c\lambda t^2 \right)
\end{equation}
for $t \in A_2$. Moreover, $|G(x)-G(xe^{it})| \ll b|t|$, yielding
\begin{equation}
	|I_2| \ll b \int_{A_2} |t| \exp\left( -c\lambda t^2 \right)\, dt \ll b \exp\left( -c\lambda t_0^2 \right) \int_{A_2} |t| \, dt \ll \frac{b}{m^2}\exp\left( -c\lambda t_0^2 \right) \ll \frac{1}{\sqrt{\lambda}} \frac{b}{n}. 
\end{equation}
We turn our attention to $I_3$. We first treat the case $m \ge 4$. We have $|G_q(xe^{it})-G_q(x)| \ll b$ for $|t| \le \pi$. By \cite[Lem.~12]{manstavicius2016},
\begin{equation}
	\max_{1/m \le |t| \le \pi} \Re \sum_{j=1}^{m} \frac{x^j (e^{itj} - 1)}{j} \le -\frac{1}{4\pi^2} \frac{u^{1-\frac{4}{m+1}}}{\log^2 u} + \frac{2}{m} + \frac{2}{\log u}
\end{equation}
for $n/m \ge 3$. Hence
\begin{equation}
	\begin{split}
		|I_3| &\ll b \max_{1/m \le |t| \le \pi} \left|\exp\left( \sum_{j=1}^{m} \frac{x^j (e^{itj} - 1)}{j} - int\right)\right|  \\
		& \ll b \exp\left( -\frac{1}{4\pi^2} \frac{u^{1-\frac{4}{m+1}}}{\log^2 u} + \frac{2}{m} + \frac{2}{\log u} \right) \\
		& =\frac{ b}{n\sqrt{\lambda}} \exp\left(\frac{\log \lambda}{2}+\log n -\frac{1}{4\pi^2} \frac{u^{1-\frac{4}{m+1}}}{\log^2 u} + \frac{2}{m} + \frac{2}{\log u} \right).
	\end{split}
\end{equation} 
We want to show that the expression in the exponent is bounded by a constant from above, and then the upper bound for $|I_3|$ will match that of $|I_1|$. The terms $2/m$ and $2/\log u$ are already bounded. The term $\log \lambda /2$ is bounded by $C + (\log mn)/2$ by Lemma~\ref{lem:lambda}, so it suffices to bound
\begin{equation}
	S(m,u):=2\log m + \frac{3 \log u}{2} -\frac{1}{4\pi^2} \frac{u^{1-\frac{4}{m+1}}}{\log^2 u}.
\end{equation}
We may also assume $n \gg 1$, since otherwise $S(m,u)$ is trivially bounded. As $ m\ge 4$, if $4\log u> m+1$ we have
\begin{equation}
	S(m,u) \le 2\log m + \frac{3\log u}{2} -\frac{1}{4\pi^2} \frac{u^{1-\frac{4}{5}}}{\log^2 u} \le 2\log (4 \log u-1) +  \frac{3\log u}{2} -\frac{1}{4\pi^2} \frac{u^{\frac{1}{5}}}{\log^2 u},
\end{equation}
which is bounded for $u \ge 3$. If $4 \log u \le m+1$ and $n$ is sufficiently large,
\begin{equation}
	\frac{1}{4\pi^2} \frac{u^{1-\frac{4}{m+1}}}{\log^2 u} \ge \frac{cu}{\log^2 u} \ge \frac{c\log n\log^3\log(n+1)}{\log^2\log(n+1)} \ge \frac{7 \log n}{2} \ge 2\log m + \frac{3 \log u}{2}
\end{equation}
using the condition $m \le n/(\log n\log^3\log(n+1))$, and so $S(m,u) \le 0$. All in all, $S(m,u)$ is bounded. If $1 \le m \le 3$, we run the same argument but with the naive bound $\Re \sum_{j=1}^{m}x^j (e^{itj} - 1)/j \le -cn^{1/m}/m$ for $1/m \le |t| \le \pi$ coming from considering only the first term in the sum. In summary, 
\begin{equation}
	|I_3| \ll \frac{b}{n\sqrt{\lambda}}.
\end{equation}
By \eqref{eq:diff p} and our bounds on $I_i$, we find that
\begin{equation}\label{eq:diff p2}
	\left| \PP(f_n \text{ is }m\text{-smooth}) - \PP(\pi_n \text{ is }m\text{-smooth}) G_q(x) \right| \ll  \frac{D(x)}{\sqrt{\lambda}}\frac{a+bm}{nm}.
\end{equation}
By \eqref{eq:pn saddle}, $D(x)/\sqrt{\lambda} \ll \PP(\pi_n \text{ is }m\text{-smooth})$, which yields \eqref{eq:main mediumn4}. \qed
\subsection{Proof of Theorem~\ref{thm:laplace}}
For $m\ge C\sqrt{n\log n}$ this is Theorem~\ref{thm:laplace indeed}, so it is left to deal with $C\sqrt{n\log n} \ge m \ge 6 \log n$. We may assume that $n\gg 1$ due to Proposition~\ref{prop:bndd unif}. The condition $C\sqrt{n\log n} \ge m$ implies $n / (\log n \log^2 \log(n+1)), n/3 \ge m$ (for $n \gg 1$), while the condition $m\ge 6\log n$ implies $x=\Theta(1)$ and $m \gg \log u$. Lemma~\ref{lem:deriv} now says that $G_q(x)=\Theta(1)$ and also gives upper bounds on $|G_q'(x)|$ and $|G_q''(x)|$. Plugging these in Theorem~\ref{thm:saddle3}, we obtain
\begin{equation}\label{eq:gq 3delta}
	\frac{\PP(f_n \text{ is }m\text{-smooth})}{\PP(\pi_n \text{ is }m\text{-smooth})}= G_q(x) \left( 1+O\left( \frac{\log u}{mq^{\lceil \frac{m+1}{2}\rceil}}   \right)\right).
\end{equation}
By \eqref{eq:gqx}, $G_q(x) = 1+ O(u\log u/(mq^{\lceil (m+1)/2 \rceil}))$, which yields \eqref{eq:main largen}. \qed

\subsection{Proof of Theorem~\ref{thm:saddle}}
The condition $8\log n\ge m$ implies $n / (\log n \log^2 \log(n+1)), n/3 \ge m$ as long as $n$ is sufficiently large. In this case, we have by Lemma~\ref{lem:deriv} that $1 \le G_q(x) \le C_{\varepsilon}$ and also upper bounds on $|G_q'(x)|$ and $|G_q''(x)|$. Plugging these in Theorem~\ref{thm:saddle3}, we obtain
\begin{equation}\label{eq:gq 2delta}
	\frac{\PP(f_n \text{ is }m\text{-smooth})}{\PP(\pi_n \text{ is }m\text{-smooth})}= G_q(x) \left( 1+O_{\varepsilon}\left( \frac{x^{1+\parity}}{q^{\lceil \frac{m+1}{2}\rceil}}   \right)\right).
\end{equation}
By \eqref{eq:gqx}, $G_q(x) = 1+ O_{\varepsilon}(ux^{1+\parity}/q^{\lceil (m+1)/2 \rceil})$, which gives \eqref{eq:main mediumn2}. For bounded $n$, the required result is immediate from \eqref{eq:bndd} and \eqref{eq:gqx2}. \qed

\subsection{Proof of Theorem~\ref{thm:saddle2}}
The estimate \eqref{eq:main mediumn32} is \eqref{eq:gq 2delta} if $m \le 8\log n$ and is \eqref{eq:gq 3delta} otherwise. Estimate \eqref{eq:main mediumn3} follows from Theorem~\ref{thm:saddle3} and Lemma~\ref{lem:small m}. It is useful to note that if $m-2\log_q n \ge 1/\log q$ then $1-n^{2/m}/q \ge c/m$ and if $m-2\log_qn \le -1/\log q$ then $1-q/n^{2/m} \ge c/m$. \qed

\subsection{Proof of Theorem~\ref{thm:Gqsize}}
This follows directly from the estimates for $\log G_q(x)$ given in Lemma~\ref{lem:small m}. \qed 

\section{Proof of Theorem~\ref{thm:golomb}}
We shall utilize identity \eqref{eq:iden expected}. We shall also use the fact that $\PP(f_n \text{ is }m\text{-smooth}) \ge \PP(\pi_n \text{ is }m\text{-smooth})$, see Proposition~\ref{prop:bndd unif}. For $m=1$ we have strict inequality. 

We first assume $n\log n \le \log q$. In this range we must prove $c/q^C \le \EE L(\pi_n)-\EE  L_q(f_n) \le C/q^c$. Since $L(\pi_n)$ and $L_q(f_n)$ are bounded from above by $n$, we can bound the difference $\EE L(\pi_n)-\EE L_q(f_n)$ by $n$ times the total variation distance of the two random variables, which is known to be $O(1/q)$ \cite[Thm.~6.1]{arratia1993} (cf.~\cite{barysoroker2018}), so we have
\begin{equation}
\EE L(\pi_n)-	\EE L_q(f_n) = O\left( \frac{n}{q}\right).
\end{equation}
This shows $\EE L(\pi_n)-\EE L_q(f_n) \le C\log q/ q$. To produce a lower bound, we consider the contribution of $1$-smooth polynomials and permutations to \eqref{eq:iden expected}:
\begin{multline}
	 \EE L(\pi_n) -\EE L_q(f_n)\ge \PP(f_n \text{ is }1\text{-smooth}) -  \PP(\pi_n \text{ is }1\text{-smooth}) = \frac{\binom{q+n-1}{n}}{q^n}-\frac{1}{n!}\\ = \frac{1}{n!} \left( \prod_{i=1}^{n} \left(1+\frac{i-1}{q} \right) - 1\right) \ge cn^{-n} \frac{n^2}{q} \ge \frac{c}{q^2}.
\end{multline}
We turn to the case $n \log n \ge \log q$; we may assume $n\gg 1$. To prove an upper bound, we use Theorems~\ref{thm:laplace} and \ref{thm:saddle}, and the monotonicity of $\PP(f_n \text{ is }m\text{-smooth})$ in $m$ to obtain
\begin{equation}
	\begin{split}
		\sum_{m=1}^{n}  \left( \PP(f_n \text{ is }m\text{-smooth}) - \PP(\pi_n \text{ is }m\text{-smooth}) \right) &\le M \PP(f_n \text{ is }M\text{-smooth}) + \sum_{m>M} \frac{n\PP(\pi_n \text{ is }m\text{-smooth})}{mq^{m/2}} \\
		& \ll M \PP(\pi_n \text{ is }M\text{-smooth}) + \frac{n}{q^{M/2}}
	\end{split}
\end{equation}
for any $M \ge 3\log_q n$. We take $M=\lceil A\sqrt{n \log n/\log q} \rceil$ for large $A$ (admissible if $n \gg 1$). The term $n/q^{M/2}$ is $\le C\exp(-c\sqrt{ n\log n \log q})$. By Proposition~\ref{prop:perm}, so is $M\PP(\pi_n \text{ is }M\text{-smooth})$.

We now prove a lower bound for $\EE L(\pi_n)-\EE L_q(f_n)$. Considering only the term $m=M$ in \eqref{eq:iden expected}, and using Theorem~\ref{thm:saddle2}, we obtain
\begin{equation}\label{eq:ee diff}
	\EE L(\pi_n)-\EE L_q(f_n)  \ge \PP(\pi_n \text{ is }M\text{-smooth}) \left( G_q(x) - 1+ O\left( \frac{x^{1+\parity}\min\{1, \frac{\log u}{M}\}}{q^{\lceil \frac{M+1}{2}\rceil}} \right)\right).
\end{equation} 
The term $\PP(\pi_n \text{ is }M\text{-smooth})$ is $\ge c\exp(-C\sqrt{n \log n \log q})$ by Proposition~\ref{prop:perm}. By \eqref{eq:gqx}, 
\begin{equation}\label{eq:gq1 lower}
	G_q(x) - 1 \ge c\frac{u x^{1+\parity}\min\{1, \frac{\log u}{M}\} }{q^{\lceil \frac{M+1}{2}\rceil}}.
\end{equation}
As the right-hand side of \eqref{eq:gq1 lower} dominates the error term in \eqref{eq:ee diff} (if $n \gg 1$) and $1/q^{\lceil (M+1)/2\rceil}$ is bounded from below by $c\exp(-C\sqrt{n \log n \log q})$, we conclude the proof. \qed

\section{Deriving Soundararajan's result from Ford's result}\label{sec:sound}
\subsection{Proof of Proposition~\ref{prop:perm}}\label{app:forsand}
We will use two results of Hildebrand on ratios of $\rho$ values. By \cite[Lem.~1]{hildebrand1984} we have
\begin{equation}\label{eq:rho ratio}
	\frac{\rho(u-t)}{\rho(u)} = e^{t\xi}(1+O(u^{-1}))
\end{equation}
uniformly for $0 \le t \le 1$ and $u \ge 1$, where $\xi\ge 0$ is as in \eqref{eq:def xi} and it satisfies $\xi \sim \log u$ by Lemma~\ref{lem:xi size}. The second result, given in \cite[Lem.~1(vi)]{hildebrand1986}, states that
\begin{equation}\label{eq:rho ratio2}
	\frac{\rho(u-t)}{\rho(u)} \ll (u \log^2 (u+1))^t
\end{equation}
uniformly for $u \ge 1$ and $0 \le t \le u$. 
By \eqref{eq:fordsand}, it suffices to show that
\[ \rho\left(\frac{n+1}{m+1} \right) - \rho\left( \frac{n}{m}\right) \ll \rho\left( \frac{n}{m}\right) \left(\exp\left( C\frac{u\log (u+1)}{m}\right)-1\right).\]
We express the left-hand side as
\[  \int_{\frac{n}{m}}^{\frac{n+1}{m+1}}\rho'(t)dt.\]
As $\rho'(u) = -\rho(u-1)/u$, this is 
\[\ll \frac{\rho\left(\frac{n}{m}\right)}{\frac{n}{m}} \int_{\frac{n+1}{m+1}}^{\frac{n}{m}} \frac{\rho(t-1)}{\rho\left(\frac{n}{m}\right)}dt.\]
If $u$ is bounded we are done, because the integral is $\ll_{u} n/m - (n+1)/(m+1) \ll_u 1/n$. We assume $u \ge 2$. Applying \eqref{eq:rho ratio2} with $(u,t) = (u-1,u-t)$, and also invoking \eqref{eq:rho ratio}, we see that the last expression is
\[\ll \frac{\rho\left(\frac{n}{m}\right)}{\frac{n}{m}} \int_{\frac{n+1}{m+1}}^{\frac{n}{m}} (u \log^2 u)^{u-t} \frac{\rho(u-1)}{\rho(u)}dt \ll  \rho(u)\log u\int_{0}^{\frac{n-m}{m(m+1)}} (u \log^2 u)^{s} ds \ll \rho(u)((u \log^2 u)^{\frac{u-1}{m+1}} - 1),\]
which is $\ll \rho(u)(\exp(Cu \log u/m)-1)$ as needed. \qed
\subsection{Proof of Theorem~\ref{thm:sound}}
By Theorems~\ref{thm:laplace} and \ref{thm:saddle}, the required estimate already holds in the range $n \ge m \ge (2+\varepsilon)\log_q n$ (with a better error term). Moreover, for bounded $n$ this is trivial, so we may assume $n \gg 1$.
It remains to prove a result in the range $3 \log_q n \ge m \ge \log_q (n\log n)$, $n \gg 1$. 
From this point on we shall assume $m$ lies in the (slightly) wider range $\log_q n < m \le 3 \log_q n$, $n \gg 1$, and we shall work out what lower bound on $m$ is needed for our result to hold. Since $\PP(f_n \text{ is }m\text{-smooth}) \ge \PP(\pi_n \text{ is }m\text{-smooth})$ by Proposition \ref{prop:bndd unif}, it suffices to show $\PP(f_n \text{ is }m\text{-smooth})/\PP(\pi_n \text{ is }m\text{-smooth})\le \exp(C u \log u / m)$.

By Theorem~\ref{thm:saddle3}, it suffices to prove the following 3 bounds: 
\begin{align}
	\log G_q(x) &= \sum_{i>m} \frac{a_i x^i}{i} \ll \frac{u \log u}{m},\\
	\frac{G_q'(x) x}{nG_q(x)} &= \sum_{i>m} \frac{a_i x^i}{n}  \ll \exp\left( \frac{Cu \log u}{m}\right),\\
	\frac{G_q''(x) x^2}{nmG_q(x)} &= \frac{1}{nm}\left(\sum_{i>m} a_i x^i \right)^2 + \frac{1}{nm}\sum_{i>m} (i-1)a_ix^i   \ll \exp\left( \frac{Cu \log u}{m}\right).
\end{align}
In the range $\log_q n < m \le 3 \log_q n$ we have $u\log u/m \asymp (n/\log n) \log^2 q$. In particular, $\exp(C u \log u/m)$ can absorb any power of $n$, $m$ and $q$. This allows us to reduce the last 3 inequalities to the following:
\begin{equation}\label{eq:3bds}
	\begin{split}
		\sum_{2m \ge i>m} a_i x^i &\ll n \log q,\\
		\sum_{i> 2m} \frac{a_i x^i}{i} &\ll  \frac{n \log^2 q}{\log n},\\
		\sum_{i > 2m} i a_i x^i&\ll \exp\left(\frac{C n \log^2 q }{\log n}\right).
	\end{split}
\end{equation}
We will now use freely the bounds $a_i \ll q^{-i/2}, q^{m-i}$, given in \eqref{eq:ai size}, as well as the estimate $x \le n^{1/m} < q$. To prove the first inequality in \eqref{eq:3bds} it suffices to show
\[\sum_{2m \ge i>m} \left(\frac{x}{\sqrt{q}}\right)^i  \ll n \log q.\]
If $x \le \sqrt{q}$, this is trivial because the left-hand side is $\le m$.  If $x/\sqrt{q}$ is greater then $1.1$ then the left-hand side is $ \ll x^{2m}/q^m \ll n^2/q^m  < n$. In the remaining case $1 \le x/\sqrt{q} \le 1.1$ the left-hand side is $\ll  m  (1.1)^{2m}  \ll \log n (1.1)^{6 \log_q n} \ll n$. Here we used $6 \log_q 1.1 < 0.9<1$.

We turn our attention to the second inequality in \eqref{eq:3bds}. It suffices to show
\[\sum_{i> 2m} \frac{(x/q)^i}{i} \ll q^{-m} \frac{n \log^2 q}{\log n}.\]
The left-hand side is $\ll (1-x/q)^{-1}(x/q)^{2m+1}/m  \ll (1-x/q)^{-1} n^2 /(q^{2m} \log_q n)$, so it suffices to show  \begin{equation}\label{eq:xineq}\left(1-\frac{x}{q}\right)^{-1} \ll\frac{q^m \log q}{n}.\end{equation}
We write $m$ as $m=\log (nf)/\log q$ for $f \ge 1$. As $x \le n^{1/m}$, it suffices to show
\[ 1- \exp\left(- \frac{\log q \log f}{\log (nf)}\right) \gg \frac{1}{f \log q}\]
which evidently holds if $f \ge \log n$.
We turn our attention to the last inequality in \eqref{eq:3bds}, which will follow from\[\sum_{i>2m} \left(\frac{x}{q}\right)^i i \ll \exp\left( \frac{Cn \log^2 q}{\log n}\right).\] 
Writing $i=j+2m+1$, the left-hand side is 
\[ \ll \left(\frac{x}{q}\right)^{2m+1} \sum_{j=0}^{\infty}\left(\frac{x}{q}\right)^j (j+2m+1)  \ll \frac{n^2}{q^m} \left( 2m\left(1-\frac{x}{q}\right)^{-1} +  \left(1-\frac{x}{q}\right)^{-2}\right)\ll n^{2}\left(1-\frac{x}{q}\right)^{-2}. \]
The required inequality now follows from \eqref{eq:xineq}, which holds if $m \ge \log(n \log n)/\log q$.
\qed

\appendix 
\section*{Appendix}
\section{Limitation to approximation via the Dickman function}
We define a quantity $\Delta(n,m)$ by
\[ \PP(\pi_n \text{ is }m\text{-smooth}) = \rho(u)(1+\Delta(n,m)).\]
By \eqref{eq:pn transform}, $\Delta(n,m) = O(u \log (u+1)/m)$ if $n\log n \le m^2$. We prove a matching lower bound.
\begin{thm}\label{thm:lower}
	For $n>m$, $\Delta(n,m) > 0$. If furthermore $m \le n/2$, then $\Delta(n,m) \gg  u \log(u)/m$.
\end{thm}
The inequality $\Delta>0$ is not new, see \cite[Ex.~6]{Granville2015}. Note that if $m$ is very close to $n$, we do not expect the lower bound $\Delta(n,m) \gg u\log (u)/m$, e.g.~$\Delta(n,n-1) = O(1/n^2) = o(u\log(u+1)/m)$. 

An asymptotic result for $\Delta$, in a restricted range, was proved by Manstavi\v{c}ius and Petuchovas in \cite[Cor.~3]{manstavicius2016}. They essentially show that $1+\Delta(n,m)\sim \exp(u\xi/2m)$ for $n^{1/3+\varepsilon} \le m \le n^{1-\varepsilon}$, where $\xi$ is as in \eqref{eq:def xi}. Their proof method is different than ours and is based on expressing the main term $D(x)/\sqrt{2\pi \lambda}$ appearing in \eqref{eq:pn saddle} in terms of $\rho$. Our proof is inspired by \cite{hildebrand1984}.

See \cite{gorodetsky2022dickman} for an exploration of the inequality $\Psi(x,y) \ge x \rho(\log x/\log y)$ motivated by a question of Pomerance.
\subsection{Preparatory lemmas}
Let
\[ S_1(n,m) := \frac{1}{u\rho(u)} \left(\frac{1}{m} \sum_{i=1}^{m} \rho\left(u-\frac{i}{m}\right) - \int_{0}^{1} \rho(u-t)dt\right)=\frac{1}{n} \sum_{i=1}^{m} \frac{\rho(u-\frac{i}{m})}{\rho(u)} - 1,\]
where the second equality follows from the identity $\int_{0}^{1} \rho(u-t) dt = u\rho(u)$. Let \[S_2(n,m):= \frac{1}{n} \sum_{i=1}^{m} \frac{\rho(u-\frac{i}{m}) \Delta(n-i,m)}{\rho(u)}.\]
\begin{lem}\label{lem:s1s2}
	For $n\ge m \ge 1$ we have the relation $\Delta(n,m) = S_1(n,m) + S_2(n,m)$.
\end{lem}
\begin{proof}
	By differentiating \eqref{eq:F form} and equating coefficients we have, for $n \ge m$,
	\[ n\PP(\pi_n \text{ is }m\text{-smooth}) = \sum_{i=1}^{m} \PP(\pi_{n-i} \text{ is }m\text{-smooth}).\]
	Rewriting this in terms of $\Delta(n,m)$ yields the result.
\end{proof}
\begin{lem}\label{lem:pos}
	If $n\ge m \ge 1$ we have $S_1(n,m)\ge 0$ with equality if and only if $n=m$. 
\end{lem}
\begin{proof}
	Since $\rho$ is strictly decreasing for $u \ge 1$, for each $1\le i\le m$ we have \[\frac{1}{m}\rho\left(u-\frac{i}{m}\right) \ge \int_{(i-1)/m}^{i/m} \rho(u-t)\,dt\]
	with equality if and only if $u \le (i-1)/m +1$. The result follows by summing over $i=1,\ldots,m$. 
\end{proof}

\begin{lem}\label{lem:strong}
	For positive $x$ let $F(x)=(1-e^{-x})^{-1} - x^{-1} \ge 1/2$. For $m \le n/2$ we have \[S_1(n,m) =(1+o_{u \to \infty}(1))\frac{\log u}{m} F\left(\frac{\log(u\log u)}{m}\right).\]
	Moreover, the $1+o_{u \to \infty}(1)$ term is bounded away from $0$ (that is, there is a uniform lower bound).  
\end{lem}
\begin{proof}
	We write
	\[ S_1(n,m) = \frac{1}{u\rho(u)} \sum_{i=1}^{m} \int_{(i-1)/m}^{i/m} \left(\rho\left( u - \frac{i}{m}\right) - \rho(u-t)\right) dt = \frac{-1}{u\rho(u)} \sum_{i=1}^{m} \int_{(i-1)/m}^{i/m} \int_{t}^{i/m}\rho'(u-x) dx\,dt.\]
	As $\rho'(u)=-\rho(u-1)/u$ and $\rho(u-1)>0$, this becomes
	\[ S_1(n,m) = \frac{1+O(u^{-1})}{u^2\rho(u)} \sum_{i=1}^{m} \int_{(i-1)/m}^{i/m} \int_{t}^{i/m}\rho\left( u -  x-1\right) dx\,dt.\]
	By applying \eqref{eq:rho ratio} to $\rho(u-x-1)/\rho(u-1)$ and to $\rho(u-1)/\rho(u)$, this is
	\begin{align} S_1(n,m) &= (1+o_{u \to \infty}(1)) \frac{\log u}{u} \sum_{i=1}^{m} \int_{(i-1)/m}^{i/m} \int_{t}^{i/m}(u \log u)^{x} dx\,dt\\
		&= \frac{1+o_{u \to \infty}(1)}{\log(u \log u)} \frac{\log u}{u} \sum_{i=1}^{m} \int_{(i-1)/m}^{i/m} ((u\log u)^{i/m} - (u \log u)^{t}) dt\\
		&= \frac{1+o_{u \to \infty}(1)}{u} \left( \frac{u\log u-1}{m(1-(u\log u)^{-1/m})} - \frac{u\log u-1}{\log(u\log u)}\right)\\
		&= (1+o_{u \to \infty}(1))\frac{\log u}{m} F\left( \frac{\log(u\log u)}{m} \right)
	\end{align}
	and the estimate follows. Running the proof when $u\ge 2$ is bounded shows that the term $1+o_{u \to \infty}(1)$ is $\ge c$.
\end{proof}

\subsection{Proof of Theorem~\ref{thm:lower}} 
Positivity follows by direct induction from Lemmas~\ref{lem:s1s2} and \ref{lem:pos}. Suppose $n/2 \ge m$. We introduce \[a(n) := \frac{\Delta(n,m)}{\log(n/m)/m}\]
which by the recurrence in Lemma~\ref{lem:s1s2} and the estimates \eqref{eq:rho ratio} and Lemma~\ref{lem:strong} satisfies the relation
\[ a(n) \ge  c + \frac{(1+o_{u \to \infty}(1))}{n}\sum_{i=1}^{m} a(n-i) (u\log u)^{i/m}\]
for $n/2 \ge m$. Moreover, $a(n) \gg 1$ by Lemma~\ref{lem:strong} and the non-negativity of $S_2$. 

Observe that $\sum_{i=1}^{m} (u\log u)^{i/m} \ge n(1+o_{u\to\infty}(1))$. Hence, if $a(n),\ldots,a(n+m-1) \ge A$ then
$a(n+m),\ldots,a(n+2m-1) \ge c + A(1+o_{u\to\infty}(1))\ge (c/2)+A$ for $u \gg 1$. Iterating this implication yields $a(n) \gg n/m$ for $u \gg 1$, implying  $\Delta(n,m) \gg u (\log u)/m$, as needed. \qed

\section*{Acknowledgments}
The author thanks Dor Elboim for useful discussions, Kevin Ford and Brad Rodgers for helpful comments on an earlier version of the manuscript and Kannan Soundararajan for kindly sharing his unpublished manuscript. The author would also like to thank the referee for their thorough and careful reading of the manuscript and useful comments which improved the presentation of the results and proofs.

\bibliographystyle{amsplain}

\begin{thebibliography}{{Gon}44}
	\bibitem[ABT93]{arratia1993}
	Richard Arratia, A.~D. Barbour, and Simon Tavar\'{e}.
	\newblock On random polynomials over finite fields.
	\newblock {\em Math. Proc. Cambridge Philos. Soc.}, 114(2):347--368, 1993.
	
	\bibitem[ABT03]{arratia2003}
	Richard Arratia, A.~D. Barbour, and Simon Tavar\'{e}.
	\newblock {\em Logarithmic combinatorial structures: a probabilistic approach}.
	\newblock EMS Monographs in Mathematics. European Mathematical Society (EMS),
	Z\"{u}rich, 2003.
	
	\bibitem[BSG18]{barysoroker2018}
	Lior Bary-Soroker and Ofir Gorodetsky.
	\newblock Roots of polynomials and the derangement problem.
	\newblock {\em Amer. Math. Monthly}, 125(10):934--938, 2018.
	
	\bibitem[Car87]{car1987}
	Mireille Car.
	\newblock Th\'{e}or\`emes de densit\'{e} dans {${\bf F}_q[X]$}.
	\newblock {\em Acta Arith.}, 48(2):145--165, 1987.
	
	\bibitem[CHM51]{chowla1951}
	S.~Chowla, I.~N. Herstein, and W.~K. Moore.
	\newblock On recursions connected with symmetric groups. {I}.
	\newblock {\em Canad. J. Math.}, 3:328--334, 1951.
	
	\bibitem[dB51]{debruijn1951}
	N.~G. de~Bruijn.
	\newblock On the number of positive integers {$\leq x$} and free of prime
	factors {$>y$}.
	\newblock {\em Nederl. Acad. Wetensch. Proc. Ser. A.}, 54:50--60, 1951.
	
	\bibitem[{Dic}30]{dickman1930}
	Karl {Dickman}.
	\newblock {On the frequency of numbers containing prime factors of a certain
		relative magnitude.}
	\newblock {\em {Ark. Mat. Astron. Fys.}}, 22 A(10):14, 1930.
		
	\bibitem[EG22]{elboim2020uniform}
	Dor Elboim and Ofir Gorodetsky.
	\newblock Uniform estimates for almost primes over finite fields.
	\newblock {\em Proc. Amer. Math. Soc.}, 150(7):2807--2822, 2022.
		
	\bibitem[For22]{ford2021cycle}
	Kevin Ford.
	\newblock Cycle type of random permutations: a toolkit.
	\newblock {\em Discrete Anal.}, 2022:9, 36 pp.
		
	\bibitem[GG97]{golomb1997}
	S.~W. Golomb and P.~Gaal.
	\newblock On the number of permutations of {$n$} objects with greatest cycle
	length {$k$}.
	\newblock In {\em Probabilistic methods in discrete mathematics
		({P}etrozavodsk, 1996)}, pages 211--218. VSP, Utrecht, 1997.
		
	\bibitem[GHS15]{Granville2015}
	Andrew Granville, Adam~J. Harper, and Kannan Soundararajan.
	\newblock Mean values of multiplicative functions over function fields.
	\newblock {\em Res. Number Theory}, 1:Paper No. 25, 18, 2015.


	\bibitem[{Gon}44]{goncharov1944}
	V.~{Goncharov}.
	\newblock {\"Uber das Gebiet der kombinatorischen Analysis}.
	\newblock {\em {Izv. Akad. Nauk SSSR, Ser. Mat.}}, 8:3--48, 1944.
		
	\bibitem[Gor17]{gorodetsky2017}
	Ofir Gorodetsky.
	\newblock A polynomial analogue of {L}andau's theorem and related problems.
	\newblock {\em Mathematika}, 63(2):622--665, 2017.


	\bibitem[Gor22a]{gorodetsky2022debruijn}
	Ofir Gorodetsky.
	\newblock Smooth integers and de {B}ruijn's approximation {$\Lambda$}.
	\newblock {\em Preprint}, 2022.
	
	\bibitem[Gor22b]{gorodetsky2022dickman}
	Ofir Gorodetsky.
	\newblock Smooth integers and the {D}ickman {$\rho$} function.
	\newblock {\em To appear in Journal d'Analyse Math\'{e}matique}.

	\bibitem[Gor22]{revisited}
	Ofir Gorodetsky.
	\newblock Smooth permutations and polynomials revisited.
	\newblock {\em Preprint}, 2022.

	\bibitem[Gra08]{granville2008}
	Andrew Granville.
	\newblock Smooth numbers: computational number theory and beyond.
	\newblock In {\em Algorithmic number theory: lattices, number fields, curves
		and cryptography}, volume~44 of {\em Math. Sci. Res. Inst. Publ.}, pages
	267--323. Cambridge Univ. Press, Cambridge, 2008.
	
	\bibitem[Hil84]{hildebrand1984}
	Adolf Hildebrand.
	\newblock Integers free of large prime factors and the {R}iemann hypothesis.
	\newblock {\em Mathematika}, 31(2):258--271 (1985), 1984.
	
	\bibitem[Hil86]{hildebrand1986}
	Adolf Hildebrand.
	\newblock On the number of positive integers {$\leq x$} and free of prime
	factors {$>y$}.
	\newblock {\em J. Number Theory}, 22(3):289--307, 1986.
	
	\bibitem[HT86]{hildebrand19862}
	Adolf Hildebrand and G\'{e}rald Tenenbaum.
	\newblock On integers free of large prime factors.
	\newblock {\em Trans. Amer. Math. Soc.}, 296(1):265--290, 1986.
	
	\bibitem[HT93]{hildebrand1993}
	A.~Hildebrand and G.~Tenenbaum.
	\newblock Integers without large prime factors.
	\newblock {\em J. Th\'{e}or. Nombres Bordeaux}, 5(2):411--484, 1993.
	
	\bibitem[JL06]{joux2006}
	Antoine Joux and Reynald Lercier.
	\newblock The function field sieve in the medium prime case.
	\newblock In {\em Advances in cryptology---{EUROCRYPT} 2006}, volume 4004 of
	{\em Lecture Notes in Comput. Sci.}, pages 254--270. Springer, Berlin, 2006.
	
	\bibitem[KM97]{knopfmacher1997}
	A.~Knopfmacher and E.~Manstavi{\v{c}}ius.
	\newblock On the largest degree of an irreducible factor of a polynomial in
	{{\(\mathbb{F}_q[X]\)}}.
	\newblock {\em Lith. Math. J.}, 37(1):38--45, 1997.

	\bibitem[LP98]{lovorn1998}
	Renet~Lovorn Bender and Carl Pomerance.
	\newblock Rigorous discrete logarithm computations in finite fields via smooth
	polynomials.
	\newblock In {\em Computational perspectives on number theory ({C}hicago, {IL},
	1995)}, volume~7 of {\em AMS/IP Stud. Adv. Math.}, pages 221--232. Amer.
	Math. Soc., Providence, RI, 1998.
	
	\bibitem[Lov92]{lovorn1992}
	Renet Lovorn.
	\newblock {\em Rigorous, subexponential algorithms for discrete logarithms over
		finite fields}.
	\newblock ProQuest LLC, Ann Arbor, MI, 1992.
	\newblock Thesis (Ph.D.)--University of Georgia.
		
	
	\bibitem[Man92a]{manstavicius1992}
	E.~Manstavi\v{c}ius.
	\newblock Remarks on elements of semigroups that are free of large prime
	factors.
	\newblock {\em Liet. Mat. Rink.}, 32(4):512--525, 1992.
	
	\bibitem[Man92b]{manstavicius19922}
	E.~Manstavi\v{c}ius.
	\newblock Semigroup elements free of large prime factors.
	\newblock In {\em New trends in probability and statistics, {V}ol. 2
		({P}alanga, 1991)}, pages 135--153. VSP, Utrecht, 1992.
	
	\bibitem[MP16]{manstavicius2016}
	Eugenijus Manstavi\v{c}ius and Robertas Petuchovas.
	\newblock Local probabilities for random permutations without long cycles.
	\newblock {\em Electron. J. Combin.}, 23(1):Paper 1.58, 25, 2016.
		
	\bibitem[MW55]{moser1955}
	Leo Moser and Max Wyman.
	\newblock On solutions of {$x^d=1$} in symmetric groups.
	\newblock {\em Canadian J. Math.}, 7:159--168, 1955.
		
	\bibitem[Odl85]{odlyzko1985}
	A.~M. Odlyzko.
	\newblock Discrete logarithms in finite fields and their cryptographic
	significance.
	\newblock In {\em Advances in cryptology ({P}aris, 1984)}, volume 209 of {\em
		Lecture Notes in Comput. Sci.}, pages 224--314. Springer, Berlin, 1985.
		
	\bibitem[Odl94]{Odlyzko1993}
	A.~M. Odlyzko.
	\newblock Discrete logarithms and smooth polynomials.
	\newblock In {\em Finite fields: theory, applications, and algorithms ({L}as
		{V}egas, {NV}, 1993)}, volume 168 of {\em Contemp. Math.}, pages 269--278.
	Amer. Math. Soc., Providence, RI, 1994.
		
	\bibitem[Odl00]{Odlyzko2000}
	Andrew Odlyzko.
	\newblock Discrete logarithms: the past and the future.
	\newblock volume~19, pages 129--145. 2000.
	\newblock Towards a quarter-century of public key cryptography.
		
	\bibitem[PGF98]{panario1998}
	Daniel Panario, Xavier Gourdon, and Philippe Flajolet.
	\newblock An analytic approach to smooth polynomials over finite fields.
	\newblock In {\em Algorithmic number theory ({P}ortland, {OR}, 1998)}, volume
	1423 of {\em Lecture Notes in Comput. Sci.}, pages 226--236. Springer,
	Berlin, 1998.
		
	\bibitem[Pol13]{pollack2013}
	Paul Pollack.
	\newblock Irreducible polynomials with several prescribed coefficients.
	\newblock {\em Finite Fields Appl.}, 22:70--78, 2013.
		
	\bibitem[Pom95]{pom}
	Carl Pomerance.
	\newblock The role of smooth numbers in number theoretic algorithms.
	\newblock In {\em Proceedings of the international congress of mathematicians,
		ICM '94, August 3-11, 1994, Z\"urich, Switzerland. Vol. I}, pages 411--422.
	Basel: Birkh{\"a}user, 1995.
			
	\bibitem[Sai89]{Saias1989}
	{\'E}ric Saias.
	\newblock Sur le nombre des entiers sans grand facteur premier.
	\newblock {\em J. Number Theory}, 32(1):78--99, 1989.
		
	\bibitem[Sch02]{Schirokauer2002}
	Oliver Schirokauer.
	\newblock The special function field sieve.
	\newblock {\em SIAM J. Discrete Math.}, 16(1):81--98, 2002.
		
	\bibitem[SL66]{shepp1966}
	L.~A. Shepp and S.~P. Lloyd.
	\newblock Ordered cycle lengths in a random permutation.
	\newblock {\em Trans. Amer. Math. Soc.}, 121:340--357, 1966.
		
	\bibitem[Sou]{soundararajan}
	K.~Soundararajan.
	\newblock Smooth polynomials: analogies and asymptotics.
	\newblock Unpublished manuscript.
		
	\bibitem[Sta99]{stanley1999enumerative}
	Richard~P. Stanley.
	\newblock {\em Enumerative combinatorics. {V}ol. 2}, volume~62 of {\em
		Cambridge Studies in Advanced Mathematics}.
	\newblock Cambridge University Press, Cambridge, 1999.
	\newblock With a foreword by Gian-Carlo Rota and appendix 1 by Sergey Fomin.
		
	\bibitem[Ten01]{tenenbaumreview}
	G\'{e}rald Tenenbaum.
	\newblock Review of the article ``{A}n analytic approach to smooth polynomials
	over finite fields'' by {D}. {P}anario, {X}. {G}ourdon and {P}. {F}lajolet.
	\newblock {\em Mathematical Reviews}, MR1726074, 2001.
		
	\bibitem[War91]{warlimont1991}
	R.~Warlimont.
	\newblock Arithmetical semigroups. {II}. {S}ieving by large and small prime
	elements. {S}ets of multiples.
	\newblock {\em Manuscripta Math.}, 71(2):197--221, 1991.
				
	\bibitem[WZ85]{wimp1985}
	Jet Wimp and Doron Zeilberger.
	\newblock Resurrecting the asymptotics of linear recurrences.
	\newblock {\em J. Math. Anal. Appl.}, 111(1):162--176, 1985.
		
	\end{thebibliography}


\begin{dajauthors}
\begin{authorinfo}[ofir]
  Ofir Gorodetsky\\
  Mathematical Institute, University of Oxford\\
  Oxford, OX2 6GG, UK\\
  ofir\imagedot{}goro\imageat{}gmail\imagedot{}com 
\end{authorinfo}
\end{dajauthors}

\end{document}